\documentclass[11pt,twoside]{article}
\usepackage{amsmath, amsthm, amscd, amsfonts, amssymb, graphicx, color}

\date{}

\begin{document}

\centerline{}

\centerline {\Large{\bf Stability of dual $g$-fusion frames in Hilbert spaces }}

\newcommand{\mvec}[1]{\mbox{\bfseries\itshape #1}}
\centerline{}
\centerline{\textbf{Prasenjit Ghosh}}
\centerline{Department of Pure Mathematics, University of Calcutta,}
\centerline{35, Ballygunge Circular Road, Kolkata, 700019, West Bengal, India}
\centerline{e-mail: prasenjitpuremath@gmail.com}
\centerline{}
\centerline{\textbf{T. K. Samanta}}
\centerline{Department of Mathematics, Uluberia College,}
\centerline{Uluberia, Howrah, 711315,  West Bengal, India}
\centerline{e-mail: mumpu$_{-}$tapas5@yahoo.co.in}

\newtheorem{Theorem}{\quad Theorem}[section]

\newtheorem{definition}[Theorem]{\quad Definition}

\newtheorem{theorem}[Theorem]{\quad Theorem}

\newtheorem{remark}[Theorem]{\quad Remark}

\newtheorem{corollary}[Theorem]{\quad Corollary}

\newtheorem{note}[Theorem]{\quad Note}

\newtheorem{lemma}[Theorem]{\quad Lemma}

\newtheorem{example}[Theorem]{\quad Example}

\newtheorem{result}[Theorem]{\quad Result}
\newtheorem{conclusion}[Theorem]{\quad Conclusion}

\newtheorem{proposition}[Theorem]{\quad Proposition}

\begin{abstract}
\textbf{\emph{We give a characterization of K-g-fusion frames and discuss the stability of dual g-fusion frames.\,We also present a necessary and sufficient condition for a quotient operator to be bounded.}}
\end{abstract}
{\bf Keywords:}  \emph{g-fusion frame, K-g-fusion frame, stability of a frame, quotient operator.}

{2020 MSC:} \emph{\text{Primary} 42C15; Secondary 46B15, 46C07.}

\section{Introduction}
 
\smallskip\hspace{.6 cm} Frames in Hilbert spaces have many remarkable properties which make them very useful in processing of signals and images, filter bank theory, coding and communications, system modeling and many other fields.\;The notion of frame in Hilbert space was born in 1952 in the work of Duffin and Schaeffer \cite{Duffin} and their idea did not appear to make much general interest outside of non-harmonic Fourier series.\;Later on,
after some innovative work of Daubechies, Grossman, Meyer \cite{Daubechies}, the theory of frames began to be studied more widely.

The theory of frames has been generalized rapidly and various generalizations of frames in Hilbert spaces namely, \,$K$-frames, \,$G$-frames, fusion frames etc.\,have been introduced in recent times.\;$K$-frames in Hilbert spaces were introduced by L.\,Gavruta \cite{Laura} to study the atomic system relative to a bounded linear operator.\;Ramu and Johnson \cite{Ramu} obtained characterizations connecting \,$K$-frames and quotient operators.\;Sun \cite{Sun} introduced a \,$g$-frame and a \,$g$-Riesz basis in complex Hilbert space and discussed several properties of them.\;$g$-frames were also defined by Kaftal, Larson, Zhang (\cite{Kaftal}).\;Huang \cite{Hua} began to study \,$K$-$g$-frame by combining \,$K$-frame and \,$g$-frame.\;General frame theory of subspaces were introduced by P.\,Casazza and G.\,Kutynoik \cite{Kutyniok} as a natural generalization of the frame theory in Hilbert spaces.\;Fusion frames and \,$K$-frames are the special case of generalized frames.\;Construction of \,$K$-$g$-fusion frames and their dual were presented by Sadri and Rahimi \cite{Sadri} to generalize the theory of \,$K$-frame, fusion frame and \,$g$-frame.\;In the theorey of frames, the stability of a frames is very important concept.\;The stability of \,$g$-frames and their dual \,$g$-frames have been studied by W.\,Sun \cite{W} and proved that if two \,$g$-frames are closed to each other, so their dual \,$g$-frames are also closed to each other.

In this paper, we study the stability of dual\;$g$-fusion frames and see that dual\;$g$-fusion frames are stable under small perturbation.\;Also, we give a characterization of \,$K$-$g$-fusion frames and at the end, we establish that a quotient operator will be bounded if and only if a \,$g$-fusion frame becomes \,$U\,K$-$g$-fusion frame.

Throughout this paper,\;$H$\, is considered to be a separable Hilbert space with associated inner product \,$\left <\,\cdot \,,\, \cdot\,\right>$\, and \,$I_{H}$\; is the identity operator on \,$H$.\;We denote the collection of all bounded linear operators from \,$H_{\,1} \;\text{to}\; H_{\,2}$\; by \,$\mathcal{B}\,(\,H_{\,1},\, H_{\,2}\,)$, where \,$H_{\,1},\, H_{\,2}$\; are two Hilbert spaces.\;In particular \,$\mathcal{B}\,(\,H\,)$\; denote the space of all bounded linear operators on \,$H$.\;For \,$T \,\in\, \mathcal{B}\,(\,H\,)$, we denote \,$\mathcal{N}\,(\,T\,)$\; and \,$\mathcal{R}\,(\,T\,)$\; for null space and range of \,$T$, respectively.\;Also \;$P_{\,V} \,\in\, \mathcal{B}\,(\,H\,)$\; is the orthogonal projection onto a closed subspace \,$V \,\subset\, H$.\,$I,\, J$\; will denote countable index sets and \,$\left\{\,H_{j}\,\right\}_{ j \,\in\, J}$\; is a sequence of Hilbert spaces.\;Define \,$l^{\,2}\left(\,\left\{\,H_{j}\,\right\}_{ j \,\in\, J}\,\right)$\; by
\[l^{\,2}\left(\,\left\{\,H_{j}\,\right\}_{ j \,\in\, J}\,\right) \;=\; \left \{\,\{\,f_{\,j}\,\}_{j \,\in\, J} \;:\; f_{\,j} \;\in\; H_{j}\,,\; \sum\limits_{\,j \,\in\, J}\, \left \|\,f_{\,j}\,\right \|^{\,2} \;<\; \infty \,\right\}\]
with inner product is given by 
\[\left<\,\{\,f_{\,j}\,\}_{\,j \,\in\, J} \,,\, \{\,g_{\,j}\,\}_{\,j \,\in\, J}\,\right> \;=\; \sum\limits_{\,j \,\in\, J}\, \left<\,f_{\,j} \,,\, g_{\,j}\,\right>_{\,H_{j}}.\]
Clearly \,$l^{\,2}\left(\,\left\{\,H_{j}\,\right\}_{ j \,\in\, J}\,\right)$\; is a Hilbert space with the pointwise operations \cite{Sadri}.

\section{Preliminaries}

\smallskip\hspace{.6 cm} In this section, we briefly recall some necessary definitions and results that will be needed later.

\begin{theorem}\cite{Gavruta}\label{th1}
Let \,$T$\, be a bounded linear operator on \,$H$\, and $V$\; be a closed subspace of \,$H$.\;Then \,$P_{\,V}\, T^{\,\ast} \,=\, P_{\,V}\,T^{\,\ast}\, P_{\,\overline{T\,V}}$.\;Moreover, if \,$T$\; is an unitary operator then \,$P_{\,\overline{T\,V}}\;T \,=\, T\,P_{\,V}$.
\end{theorem}

\begin{theorem}\cite{Christensen}\label{th1.01}
The set \,$\mathcal{S}\,(\,H\,)$\; of all self-adjoint operators on \,$H$\; is a partially ordered set with respect to the partial order \,$\leq$\, which is defined as for \,$T,\,S \,\in\, \mathcal{S}\,(\,H\,)$ 
\[ T \,\leq\, S \,\Leftrightarrow\, \left<\,T\,f \,,\, f\,\right> \,\leq\, \left<\,S\,f \,,\, f\,\right>\; \;\forall\; f \,\in\, H.\] 
\end{theorem}

\begin{definition}{\cite{Ramu}}
Let \,$U,\, V \,\in\, \mathcal{B}\,(\,H\,)$\; with \,$\mathcal{N}\,(\,V\,) \,\subset\, \mathcal{N}\,(\,U\,)$.\;Then the linear operator \,$T \,=\, [\;U \,/\, V\;] \,:\, \mathcal{R}\,(\,V\,) \,\rightarrow\, \mathcal{R}\,(\,U\,)$, defined by \,$T\,(\,V\,f\,) \;=\; U\,f,\; f \,\in\, H$\, is called quotient operator on \,$H$.\;It can be verify that \,$\mathcal{R}\,(\,T\,) \,\subset\, \mathcal{R}\,(\,U\,)$\, and \,$T\,V \,=\, U$. 
\end{definition}

\begin{definition}\cite{Christensen}
A frame for \,$H$\; is a sequence \,$\left\{\,f_{\,j}\,\right\}_{j \,\in\, J} \,\subseteq\, H$\;  such that
\[ A\; \|\,f\,\|^{\,2} \;\leq\; \;\sum\limits_{j \,\in\, J}\;  \left |\, \left <\,f \;,\; f_{\,j} \, \right >\, \right |^{\,2} \;\leq\; B \;\|\,f\,\|^{\,2}\; \;\forall\; f \;\in\; H \]
for some positive constants \,$A,\, B$.\,The constants \,$A$\, and \,$B$\, are called frame bounds.
\end{definition}

\begin{definition}\cite{Kutyniok}
Let \,$\left\{\,W_{j}\,\right\}_{ j \,\in\, J}$\, be a collection of closed subspaces of \,$H$\, and \,$\left\{\,v_{j}\,\right\}_{ j \,\in\, J}$\, be a collection of positive weights.\;A fusion frame for \,$H$\; is a family of weighted closed subspaces \,$\left\{\, (\,W_{j},\, v_{j}\,) \,:\, j \;\in\; J\,\right\}$\; such that
\[A \;\left \|\, f \,\right \|^{\,2} \;\leq\; \sum\limits_{\,j \,\in\, J}\, v_{j}^{\,2} \;\left\|\, P_{\,W_{j}}\,(\,f\,) \,\right\|^{\,2} \;\leq\; B \; \left\|\, f \, \right\|^{\,2}\; \;\forall\; f \;\in\; H\]
for some \,$0 \,<\, A \, \leq\, B \,<\, \infty$.\;The constants \,$A ,\, B$\; are called fusion frame bounds.
\end{definition}

\begin{definition}\cite{Kaftal,Sun}
A generalized frame or g-frame for \,$H$\, with respect to \,$\left\{\,H_{j}\,\right\}_{j \,\in\, J}$\, is a sequence of operators \,$\left\{\,\Lambda_{j} \,\in\, \mathcal{B}\,(\,H,\, H_{j}\,) \,:\, j \,\in\, J\,\right\}$\; such that
\[A \;\left \|\, f \,\right \|^{\,2} \;\leq\; \sum\limits_{\,j \,\in\, J}\, \left\|\,\Lambda_{j}\,f \,\right\|^{\,2} \;\leq\; B \; \left\|\, f \, \right\|^{\,2}\;\; \;\forall\; f \;\in\; H\]
for some positive constants \,$A$\; and \,$B$.\;The constants \,$A$\; and \,$B$\; are called the lower and upper bounds, respectively.
\end{definition}

\begin{definition}\cite{Ahmadi}
Let \,$\left\{\,W_{j}\,\right\}_{ j \,\in\, J}$\; be collection of closed subspaces of \,$H$\; and \,$\left\{\,v_{j}\,\right\}_{ j \,\in\, J}$\; be a collection of positive weights and let \,$\Lambda_{j} \,\in\, \mathcal{B}\,(\,H,\, H_{j}\,)$\; for each \,$j \,\in\, J$.\;A generalized fusion frame or a g-fusion frame for \,$H$\, with respect to \,$\left\{\,H_{j}\,\right\}_{j \,\in\, J}$\, is a family of the form \,$\Lambda \,=\, \left\{\,\left(\,W_{j},\, \Lambda_{j},\, v_{j}\,\right)\,\right\}_{j \,\in\, J}$\; such that
\begin{equation}\label{eq1}
A \;\left \|\,f \,\right \|^{\,2} \;\leq\; \sum\limits_{\,j \,\in\, J}\,v_{j}^{\,2}\, \left\|\,\Lambda_{j}\,P_{\,W_{j}}\,(\,f\,) \,\right\|^{\,2} \;\leq\; B \; \left\|\, f \, \right\|^{\,2}\; \;\forall\; f \;\in\; H,
\end{equation}
for some constants \,$0 \,<\, A \,\leq\, B \,<\, \infty$.\;The constants \,$A$\; and \,$B$\; are called the lower and upper bounds of g-fusion frame, respectively.\;If \,$A \,=\, B$\, then \,$\Lambda$\; is called tight g-fusion frame and if \,$A \,=\, B \,=\,  1$\; then we say \,$\Lambda$\; is a Parseval g-fusion frame.\;If the family \,$\Lambda$\; satisfies
\[\sum\limits_{\,j \,\in\, J}\,v_{j}^{\,2}\, \left\|\,\Lambda_{j}\,P_{\,W_{j}}\,(\,f\,) \,\right\|^{\,2} \;\leq\; B \; \left\|\, f \, \right\|^{\,2}\; \;\forall\; f \;\in\; H\]then it is called a g-fusion Bessel sequence for \,$H$\, with a bound \,$B$. 
\end{definition}

\begin{definition}\cite{Ahmadi}
Let \,$\Lambda \,=\, \left\{\,\left(\,W_{j},\, \Lambda_{j},\, v_{j}\,\right)\,\right\}_{j \,\in\, J}$\, be a g-fusion Bessel sequence for \,$H$.\;Then the operator \,$ T_{\,\Lambda} \;:\; l^{\,2}\left(\,\left\{\,H_{j}\,\right\}_{ j \,\in\, J}\,\right) \,\to\, H$\, defined by
\[T_{\Lambda}\,\left(\,\left\{\,f_{\,j}\,\right\}_{j \,\in\, J}\,\right) \,=\,  \sum\limits_{\,j \,\in\, J}\,v_{j}\, P_{\,W_{j}}\,\Lambda_{j}^{\,\ast}\,f_{\,j}\;\; \;\forall\; \{\,f_{\,j}\,\}_{j \,\in\, J} \,\in\, l^{\,2}\left(\,\left\{\,H_{j}\,\right\}_{ j \,\in\, J}\,\right)\]is called synthesis operator and the operator \,$T_{\,\Lambda}^{\,\ast} \,:\, H \,\to\, l^{\,2}\left(\,\left\{\,H_{j}\,\right\}_{ j \,\in\, J}\,\right)$\, defined by
\[T_{\,\Lambda}^{\,\ast}\,(\,f\,) \;=\;  \left\{\,v_{j}\,\Lambda_{j}\,P_{\,W_{j}}\,(\,f\,)\,\right\}_{ j \,\in\, J}\; \forall\; f \;\in\; H\]is called analysis operator.\;The operator \,$S_{\Lambda} \;:\; H \;\to\; H$\, defined by
\begin{equation}\label{eq1.1}
S_{\Lambda}\,f \;=\; T_{\,\Lambda}\,T_{\,\Lambda}^{\,\ast}\,f \;=\; \sum\limits_{\,j \,\in\, J}\,v_{j}^{\,2}\, P_{\,W_{j}}\,\Lambda_{j}^{\,\ast}\, \Lambda_{j}\, P_{\,W_{j}}\,(\,f\,)\; \;\forall\; f \;\in\; H
\end{equation}
is called g-fusion frame operator.\;It can be easily verify that
\[\left<\,S_{\Lambda}\,f \,,\, f\,\right> \,=\, \sum\limits_{\,j \,\in\, J}\,v_{j}^{\,2}\, \left\|\,\Lambda_{j}\,P_{\,W_{j}}\,(\,f\,) \,\right\|^{\,2}\; \;\forall\; f \;\in\; H.\] 
Furthermore, if \,$\Lambda$\, is a g-fusion frame with bounds \,$A$\, and \,$B$, then from (\ref{eq1}),
\[\left<\,A\,f \,,\, f\,\right> \,\leq\, \left<\,S_{\Lambda}\,f \,,\, f\,\right> \,\leq\, \left<\,B\,f \,,\, f\,\right>\; \;\forall\; f \;\in\; H.\]
The operator \,$S_{\Lambda}$\; is bounded, self-adjoint, positive and invertible.\;Now, according to the Theorem (\ref{th1.01}), we can write, \,$A\,I_{\,H} \,\leq\,S_{\Lambda} \,\leq\, B\,I_{H}$\; and this gives \[ B^{\,-1}\,I_{H} \,\leq\, S_{\,\Lambda}^{\,-1} \,\leq\, A^{\,-1}\,I_{H}.\]
\end{definition}

\begin{theorem}\cite{Ahmadi}\label{th1.1}
$\Lambda \,=\, \left\{\,\left(\,W_{j},\, \Lambda_{j},\, v_{j}\,\right)\,\right\}_{j \,\in\, J}$\; is a g-fusion Bessel sequence for \,$H$\; with bound \,$B$\; if and only if the synthesis operator \,$T_{\Lambda}$\, is a well-defined and bounded with \,$\left\|\,T_{\Lambda}\,\right\| \,\leq\, \sqrt{B}$.
\end{theorem}

\begin{definition}\cite{Sadri}
Let \,$\left\{\,W_{j}\,\right\}_{ j \,\in\, J}$\; be collection of closed subspaces of \,$H$\; and \,$\left\{\,v_{j}\,\right\}_{ j \,\in\, J}$\; be a collection of positive weights and let \,$\Lambda_{j} \,\in\, \mathcal{B}\,(\,H,\, H_{j}\,)$\; for each \,$j \,\in\, J$\; and \,$K \;\in\; \mathcal{B}\,(\,H\,)$.\;Then K-g-fusion frame for \,$H$\, with respect to \,$\left\{\,H_{j}\,\right\}_{j \,\in\, J}$\, is a family of the form \,$\Lambda \,=\, \left\{\,\left(\,W_{j},\, \Lambda_{j},\, v_{j}\,\right)\,\right\}_{j \,\in\, J}$\; such that
\begin{equation}\label{eq2}
A \;\left \|\,K^{\,\ast}\,f \,\right \|^{\,2} \;\leq\; \sum\limits_{\,j \,\in\, J}\,v_{j}^{\,2}\, \left\|\,\Lambda_{j}\,P_{\,W_{j}}\,(\,f\,) \,\right\|^{\,2} \;\leq\; B \; \left\|\, f \, \right\|^{\,2}\; \;\forall\; f \;\in\; H
\end{equation}
for some constants \,$0 \,<\, A \;\leq\; B \,<\, \infty$.\;The constants \,$A$\; and \,$B$\; are called the lower and upper bounds of K-g-fusion frame, respectively.\;If \,$A \,=\, B$\; then  \,$\Lambda$\; is called a tight K-g-fusion frame.\;If \,$K \,=\, I_{H}$\, then \,$\Lambda$\; is a g-fusion frame and if \,$K \,=\, I_{H}\, \,\text{and}\, \,\Lambda_{j} \,=\, P_{\,W_{j}}$\; for any \,$j \,\in\, J$, then \,$\Lambda$\; is a fusion frame for \,$H$.
\end{definition}

\section{Some properties of \,$K$-$g$-fusion frames}

\begin{theorem}
Let \,$U \,\in\, \mathcal{B}\,(\,H\,)$\, be an invertible operator on \,$H$\, and \,$\Lambda \,=\, \left\{\,\left(\,W_{j},\, \Lambda_{j},\, v_{j}\,\right)\,\right\}_{j \,\in\, J}$\; be a \,$K$-$g$-fusion frame for \,$H$\; for some \,$K \,\in\, \mathcal{B}\,(\,H\,)$.\;Then \;$\Gamma \,=\, \left\{\,\left(\,U\,W_{j},\, \Lambda_{j}\,P_{\,W_{j}}\,U^{\,\ast},\, v_{j}\,\right)\,\right\}_{j \,\in\, J}$\; is a \,$U\,K\,U^{\,\ast}$-g-fusion frame for \,$H$. 
\end{theorem}

\begin{proof}
Since \,$\Lambda$\, is a \,$K$-$g$-fusion frame for \,$H$, \,$\;\exists\; A,\, B \,>\, 0$\, such that 
\begin{equation}\label{eq2.1}
A \;\left \|\,K^{\,\ast}\,f \,\right \|^{\,2} \;\leq\; \sum\limits_{\,j \,\in\, J}\,v_{j}^{\,2}\, \left\|\,\Lambda_{j}\,P_{\,W_{j}}\,(\,f\,)\,\right\|^{\,2} \;\leq\; B \; \left\|\,f\, \right\|^{\,2}\; \;\forall\; f \,\in\, H.
\end{equation}
Also, \,$U$\, is an invertible bounded linear operator on \,$H$, so for any \,$j \,\in\, J$, \,$U\,W_{j}$\\ is closed in \,$H$.\;Now, for each \,$f \,\in\, H$, using Theorem (\ref{th1}), we obtain
\[\sum\limits_{\,j \,\in\, J}\, v_{j}^{\,2}\, \left\|\,\Lambda_{j}\,P_{\,W_{j}}\,U^{\,\ast}\,P_{\,U\,W_{j}}\,(\,f\,)\,\right\|^{\,2} \,=\, \sum\limits_{\,j \,\in\, J}\, v_{j}^{\,2}\, \left\|\,\Lambda_{j}\,P_{\,W_{j}}\,U^{\,\ast}\,(\,f\,)\,\right\|^{\,2}\]
\[\leq\, B\; \left\|\,U^{\,\ast}\,f\,\right\|^{\,2} \,\leq\, B\, \|\,U\,\|^{\,2}\, \|\,f\,\|^{\,2}\; \;[\;\text{by (\ref{eq2.1})}\;].\]On the other hand,
\[\dfrac{A}{\|\,U\,\|^{\,2}}\, \left\|\,\left(\,U\,K\,U^{\,\ast}\,\right)^{\,\ast}\,f\,\right\|^{\,2} \,=\, \dfrac{A}{\|\,U\,\|^{\,2}}\, \left\|\,U\,K^{\,\ast}\,U^{\,\ast}\,f\,\right\|^{\,2} \,\leq\, A\, \left\|\,K^{\,\ast}\,U^{\,\ast}\,f\,\right\|^{\,2}\]
\[\hspace{3.8cm} \,\leq\, \sum\limits_{\,j \,\in\, J}\, v_{j}^{\,2}\, \left\|\,\Lambda_{j}\,P_{\,W_{j}}\,\left(\,U^{\,\ast}\,f\,\right)\,\right\|^{\,2}\; \;[\;\text{by (\ref{eq2.1})}\;]\]
\[\hspace{5.3cm}\,=\, \sum\limits_{\,j \,\in\, J}\, v_{j}^{\,2}\, \left\|\,\Lambda_{j}\,P_{\,W_{j}}\,U^{\,\ast}\,P_{\,U\,W_{j}}\,(\,f\,)\,\right\|^{\,2}\; \;\;\forall\; f \,\in\, H.\] 
Therefore, \,$\Gamma$\; is a \,$U\,K\,U^{\,\ast}$-$g$-fusion frame for \,$H$. 
\end{proof}

\begin{theorem}
Let \,$U$\, be an invertible bounded linear operator on \,$H$\, and \,$\Gamma \,=\, \left\{\,\left(\,U\,W_{j},\, \Lambda_{j}\,P_{\,W_{j}}\,U^{\,\ast},\, v_{j}\,\right)\,\right\}_{j \,\in\, J}$\; be a K-g-fusion frame for \,$H$\, for some \,$K \,\in\, \mathcal{B}\,(\,H\,)$. Then \,$\Lambda \,=\, \left\{\,\left(\,W_{j} ,\, \Lambda_{j},\, v_{j}\,\right)\,\right\}_{j \,\in\, J}$\, is a \,$U^{\,-\, 1}\,K\,U$-g-fusion frame for \,$H$.   
\end{theorem}

\begin{proof}
Since \,$\Gamma$\, is a \,$K$-$g$-fusion frame for \,$H$, for all \,$f \,\in\, H, \;\exists\; A,\, B \,>\, 0$\, such that
\begin{equation}\label{eq2.2}
A \;\left \|\,K^{\,\ast}\,f \,\right \|^{\,2} \;\leq\; \sum\limits_{\,j \,\in\, J}\, v_{j}^{\,2}\, \left\|\,\Lambda_{j}\,P_{\,W_{j}}\,U^{\,\ast}\,P_{\,U\,W_{j}}\,(\,f\,)\,\right\|^{\,2} \,\leq\, B \; \left\|\,f\, \right\|^{\,2}.
\end{equation}
Now, for each \,$f \,\in\, H$, we have
\[\dfrac{A}{\|\,U\,\|^{\,2}}\,\left \|\,\left(\,U^{\,-\, 1}\,K\,U\,\right)^{\,\ast}\,f\,\right \|^{\,2} \,=\, \dfrac{A}{\|\,U\,\|^{\,2}}\,\left\|\,U^{\,\ast}\,K^{\,\ast}\,(\,U^{\,-\, 1}\,)^{\,\ast}\,f\,\right\|^{\,2}\]
\[ \,\leq\, A\;\left\|\,K^{\,\ast}\,\left(\,U^{\,-\, 1}\,\right)^{\,\ast}\,f\,\right\|^{\,2}\hspace{2.7cm}\]
\[\hspace{2.7cm} \leq\, \sum\limits_{\,j \,\in\, J}\, v_{j}^{\,2}\, \left\|\,\Lambda_{j}\,P_{\,W_{j}}\,U^{\,\ast}\,P_{\,U\,W_{j}}\,\left(\,\left(\,U^{\,-\, 1}\,\right)^{\,\ast}\,f\,\right)\,\right\|^{\,2}\; [\;\text{by (\ref{eq2.2})}\;]\]
\[\hspace{3.5cm} \,=\, \sum\limits_{\,j \,\in\, J}\, v_{j}^{\,2}\, \left\|\,\Lambda_{j}\,P_{\,W_{j}}\,\left(\,U^{\,\ast}\,\left(\,U^{\,-\, 1}\,\right)^{\,\ast}\,f\,\right)\,\right\|^{\,2}\; [\;\text{by Theorem (\ref{th1})}\;]\] 
\[ \,=\, \sum\limits_{\,j \,\in\, J}\, v_{j}^{\,2}\, \left\|\,\Lambda_{j}\,P_{\,W_{j}}\,(\,f\,)\,\right\|^{\,2}.\hspace{2cm}\]Also, for each \,$f \,\in\, H$, we have
\[\sum\limits_{\,j \,\in\, J}\, v_{j}^{\,2}\, \left\|\,\Lambda_{j}\,P_{\,W_{j}}\,(\,f\,)\,\right\|^{\,2} \,=\, \sum\limits_{\,j \,\in\, J}\, v_{j}^{\,2}\, \left\|\,\Lambda_{j}\,P_{\,W_{j}}\,\left(\,U^{\,\ast}\,\left(\,U^{\,-\, 1}\,\right)^{\,\ast}\,f\,\right)\,\right\|^{\,2}\]
\[\hspace{5cm}\,=\, \sum\limits_{\,j \,\in\, J}\, v_{j}^{\,2}\, \left\|\,\Lambda_{j}\,P_{\,W_{j}}\,U^{\,\ast}\,P_{\,U\,W_{j}}\,\left(\,\left(\,U^{\,-\, 1}\,\right)^{\,\ast}\,f\,\right)\,\right\|^{\,2}\]
\[\hspace{2.3cm}\leq\, B \; \left\|\,\left(\,U^{\,-\, 1}\,\right)^{\,\ast}\,f\, \right\|^{\,2}\; [\;\text{by (\ref{eq2.2})}\;]\]
\[\hspace{1cm}\,\leq\, B\; \left\|\,U^{\,-\, 1}\,\right\|^{\,2}\,\|\,f\,\|^{\,2}.\]
Thus, \,$\Lambda$\, is a \,$U^{\,-\, 1}\,K\,U$-$g$-fusion frame for \,$H$\, with bounds \;$\dfrac{A}{\|\,U\,\|^{\,2}}$\, and \,$B\; \left\|\,U^{\,-\, 1}\,\right\|^{\,2}$.         
\end{proof}

\begin{theorem}\label{th3}
Let \,$K$\, be an invertible bounded linear operator on \,$H$\, and \,$\Lambda \,=\, \left\{\,\left(\,W_{j},\, \Lambda_{j},\, v_{j}\,\right)\,\right\}_{j \,\in\, J}$\; be a g-fusion frame for \,$H$\; with frame bounds \,$A,\, B$\, and \,$S_{\Lambda}$\, be the associated g-fusion frame operator.\;Then \,$\left\{\,\left(\,K\,S_{\Lambda}^{\,-\,1}\,W_{j},\; \Lambda_{j}\,P_{\,W_{j}}\,S_{\Lambda}^{\,-\,1}\,K^{\,\ast},\; v_{j}\,\right)\,\right\}_{j \,\in\, J}$\; is a K-g-fusion frame for \,$H$\, with the corresponding g-fusion frame operator \,$K\,S_{\Lambda}^{\,-\,1}\,K^{\,\ast}$.  
\end{theorem}

\begin{proof}
Let \,$T \,=\, K\,S_{\Lambda}^{\,-\,1}$.\;Then \,$T$\, is invertible on \,$H$\, and \,$T^{\,\ast} \,=\, \left(\,K\,S_{\Lambda}^{\,-\,1}\,\right)^{\,\ast} \,=\, S_{\Lambda}^{\,-\,1}\,K^{\,\ast}$.\;For \,$f \,\in\, H$, we have 
\[\left\|\,K^{\,\ast}\,f\,\right\|^{\,2} \,=\, \left\|\,S_{\Lambda}\,S_{\Lambda}^{\,-\, 1}\,\,K^{\,\ast}\,f\,\right\|^{\,2} \,\leq\, \left\|\,S_{\Lambda}\,\right\|^{\,2}\,\left\|\,S_{\Lambda}^{\,-\,1}\,K^{\,\ast}\,f\,\right\|^{\,2}\]
\begin{equation}\label{eq2.3}
\hspace{5cm}\,\leq\, B^{\,2}\, \left\|\,S_{\Lambda}^{\,-\,1}\,K^{\,\ast}\,f\,\right\|^{\,2}.
\end{equation}
Now, for each \,$f \,\in\, H$, using Theorem (\ref{th1}), we get
\[\sum\limits_{\,j \,\in\, J}\,v_{j}^{\,2}\, \left\|\,\Lambda_{j}\,P_{\,W_{j}}\,T^{\,\ast}\,P_{\,T\,W_{j}}\,(\,f\,) \,\right\|^{\,2} \,=\, \sum\limits_{\,j \,\in\, J}\,v_{j}^{\,2}\, \left\|\,\Lambda_{j}\,P_{\,W_{j}}\,(\,T^{\,\ast}\,f\,)\,\right\|^{\,2}\]
\[\,=\, \sum\limits_{\,j \,\in\, J}\,v_{j}^{\,2}\, \left\|\,\Lambda_{j}\,P_{\,W_{j}}\,\left(\,S_{\Lambda}^{\,-\,1}\,K^{\,\ast}\,f\,\right)\,\right\|^{\,2}\]
\[\hspace{2.5cm}\leq\; B\; \left\|\,S_{\Lambda}^{\,-\,1}\,K^{\,\ast}\,f\,\right\|^{\,2}\; \;[\;\text{since}\; \,\Lambda \;\text{is a $g$-fusion frame}\;]\]
\[\leq\; B\; \|\,S_{\Lambda}^{\,-\,1}\,\|^{\,2}\; \left\|\,K^{\,\ast}\,f\,\right\|^{\,2}\hspace{1.8cm}\]
\[\hspace{3.9cm}\leq\; \dfrac{B}{A^{\,2}}\; \|\,K\,\|^{\,2}\;\|\,f\,\|^{\,2}\; \;[\;\text{using}\; \,B^{\,-1}\,I_{\,H} \,\leq\, S_{\Lambda}^{\,-\,1} \,\leq\, A^{\,-1}\,I_{\,H}\;].\] On the other hand,
\[\sum\limits_{\,j \,\in\, J}\,v_{j}^{\,2}\, \left\|\,\Lambda_{j}\,P_{\,W_{j}}\,T^{\,\ast}\,P_{\,T\,W_{j}}\,(\,f\,) \,\right\|^{\,2} \,=\, \sum\limits_{\,j \,\in\, J}\,v_{j}^{\,2}\, \left\|\,\Lambda_{j}\,P_{\,W_{j}}\,\left(\,S_{\Lambda}^{\,-\,1}\,K^{\,\ast}\,f\,\right)\,\right\|^{\,2}\hspace{5cm}\]
\[\hspace{5.81cm}\geq\; A\; \left\|\,S_{\Lambda}^{\,-\,1}\,K^{\,\ast}\,f\,\right\|^{\,2} \,\geq\, \dfrac{A}{B^{\,2}}\; \left\|\,K^{\,\ast}\,f\,\right\|^{\,2}\; \;[\;\text{by (\ref{eq2.3})}\;].\]
Thus, \,$\left\{\,\left(\,K\,S_{\Lambda}^{\,-\,1}\,W_{j},\; \Lambda_{j}\,P_{\,W_{j}}\,S_{\Lambda}^{\,-\,1}\,K^{\,\ast},\; v_{j}\,\right)\,\right\}_{j \,\in\, J}$\; is a \,$K$-$g$-fusion frame for \,$H$.\\Furthermore, for each \,$f \,\in\, H$, 
\[\sum\limits_{\,j \,\in\, J}\, v_{j}^{\,2}\; P_{\,T\,W_{j}}\, \left(\,\Lambda_{j}\, P_{\,W_{j}}\,T^{\,\ast} \,\right)^{\,\ast}\, \left(\,\Lambda_{j}\,P_{\,W_{j}}\,T^{\,\ast}\,\right)\, P_{\,T\,W_{j}}\,(\,f\,)\hspace{2cm}\]
\[\hspace{.9cm}=\; \sum\limits_{\,j \,\in\, J}\, v_{j}^{\,2}\; \left(\,P_{\,T\,W_{\,j}}\, T\, P_{\,W_{j}}\,\right)\, \Lambda_{j}^{\,\ast}\; \Lambda_{j}\, \left(\,P_{\,W_{j}}\, T^{\,\ast}\, P_{\,T\,W_{j}}\,\right)\,(\,f\,)\]
\[\hspace{1cm}=\; \sum\limits_{\,j \,\in\, J}\,v_{j}^{\,2}\,\left(\,P_{\,W_{j}}\,T^{\,\ast}\,P_{\,T\,W_{j}}\,\right)^{\,\ast}\,\Lambda_{j}^{\,\ast}\;\Lambda_{j}\,(\,P_{\,W_{j}}\,T^{\,\ast}\,P_{\,T\,W_{j}}\,)\,(\,f\,)\]
\[\hspace{2.6cm}=\; \sum\limits_{\,j \,\in\, J}\, v_{j}^{\,2}\,\left(\,P_{\,W_{j}}\, T^{\,\ast}\,\right)^{\,\ast}\, \Lambda_{j}^{\,\ast}\; \Lambda_{j}\, P_{\,W_{j}}\, T^{\,\ast}\,(\,f\,)\; \;[\;\text{using Theorem (\ref{th1})}\;]\]
\[\hspace{1cm}=\; \sum\limits_{\,j \,\in\, J}\, v_{j}^{\,2}\; T\, P_{\,W_{j}}\, \Lambda_{j}^{\,\ast}\; \Lambda_{j}\, P_{\,W_{j}}\, T^{\,\ast}\,f \,=\, T\, S_{\Lambda}\, T^{\,\ast}\,f\; \;[\;\text{by (\ref{eq1.1})}\;]\]
\[\hspace{1.1cm}\,=\, \left(\,K\,S_{\Lambda}^{\,-\,1}\,\right)\,S_{\Lambda}\,\left(\,S_{\Lambda}^{\,-\,1}\,K^{\,\ast}\,f\,\right) \,=\, K\,S_{\Lambda}^{\,-\,1}\,K^{\,\ast}\,f\; \;\;\forall\; f \,\in\, H.\]This implies that \,$K\,S_{\Lambda}^{\,-\,1}\,K^{\,\ast}$\, is the associated \,$g$-fusion frame operator.
\end{proof}

\begin{corollary}
Let \,$\Lambda \,=\, \left\{\,\left(\,W_{j},\, \Lambda_{j},\, v_{j}\,\right)\,\right\}_{j \,\in\, J}$\, be a g-fusion frame for \,$H$\, with g-fusion frame operator \,$S_{\Lambda}$.\;If \,$P_{V}$\; is the orthogonal projection onto closed subspace \,$V \,\subset\, H$\, then \,$\left\{\,\left(\,P_{V}\,S_{\Lambda}^{\,-\,1}\,W_{j},\; \Lambda_{j}\,P_{\,W_{j}}\,S_{\Lambda}^{\,-\,1}\,P_{V},\; v_{j}\,\right)\,\right\}_{j \,\in\, J}$\, is a \,$P_{V}$-g-fusion frame for \,$H$\, with the corresponding g-fusion frame operator \,$P_{V}\,S_{\Lambda}^{\,-\,1}\,P_{V}$.   
\end{corollary}

\begin{proof}
Proof of this Corollary directly follows from that of the Theorem (\ref{th3}), by putting \,$K \,=\, P_{V}$.
\end{proof}

\begin{theorem}\label{thm3}
Let \,$\Lambda \,=\, \left\{\,\left(\,W_{j},\, \Lambda_{j},\, v_{j}\,\right)\,\right\}_{j \,\in\, J}$\; be a K-g-fusion frame for \,$H$\; with bounds \,$A,\, B$\, and  for each \,$j \,\in\, J,\, T_{j} \,\in\, \mathcal{B}\,\left(\,H_{j}\,\right)$\; be invertible operator.\;Suppose  
\begin{equation}\label{eq2.4}
0 \,<\, m \;=\; \inf\limits_{j \,\in\, J}\,\dfrac{1}{\left\|\,T_{j}^{\,-\, 1}\,\right\|} \;\leq\; \sup\limits_{j \,\in\, J}\,\left\|\,T_{j}\,\right\| \,=\, M.
\end{equation}
If \;$T\,\in\, \mathcal{B}\,(\,H\,)$\; is an invertible operator on \,$H$\, with \;$K\,T \,=\, T\,K$\, then \;$\Gamma \,=\, \\ \left\{\,\left(\,T\,W_{j},\; T_{j}\,\Lambda_{j}\,P_{\,W_{j}}\,T^{\,\ast},\; v_{j}\,\right)\,\right\}_{j \,\in\, J}$\; is a K-g-fusion frame for \,$H$.
\end{theorem}

\begin{proof}
Since \,$T$\, and \,$T_{j}$\, (\,for each \,$j \,\in\, J$\,) are invertible, so
\begin{equation}\label{eq2.5}
\left\|\,K^{\,\ast}\,f\,\right\|^{\,2} \,=\, \left\|\,\left(\,T^{\,-\, 1}\,\right)^{\,\ast}\,T^{\,\ast}\,K^{\,\ast}\,f\,\right\|^{\,2} \,\leq\, \left\|\,T^{\,-\, 1}\,\right\|^{\,2}\,\left\|\,T^{\,\ast}\,K^{\,\ast}\,f\,\right\|^{\,2},\, \; \;\&
\end{equation}
\begin{equation}\label{eq2.51}
\|\,f\,\|^{\,2} \,=\, \left\|\,T^{\,-\, 1}_{j}\,T_{j}\,f\,\right\|^{\,2} \,\leq\, \left\|\,T^{\,-\, 1}_{j}\,\right\|^{\,2}\; \left\|\,T_{j}\,f\,\right\|^{\,2}.
\end{equation}  
By Theorem (\ref{th1}), for each \,$f \,\in\, H$, we have 
\[\sum\limits_{\,j \,\in\, J}\,v_{j}^{\,2}\, \left\|\,T_{j}\,\Lambda_{j}\,P_{\,W_{j}}\,T^{\,\ast} \,P_{\,T\,W_{j}}\,(\,f\,) \,\right\|^{\,2} \,=\, \sum\limits_{\,j \,\in\, J}\,v_{j}^{\,2}\, \left\|\,T_{j}\,\Lambda_{j}\,P_{\,W_{j}}\,(\,T^{\,\ast}\,f\,)\,\right\|^{\,2}\]
\[\geq\; \sum\limits_{\,j \,\in\, J}\,\dfrac{1}{\left\|\,T_{j}^{\,-\, 1}\,\right\|^{\,2}}\;v_{j}^{\,2}\; \left\|\,\Lambda_{j}\,P_{\,W_{j}}\,(\,T^{\,\ast}\,f\,) \,\right\|^{\,2}\; \;[\;\text{using (\ref{eq2.51})}\;]\] 
\[\hspace{1.1cm}\geq\;m^{\,2}\;\sum\limits_{\,j \,\in\, J}\,v_{j}^{\,2}\,\left\|\,\Lambda_{j}\,P_{\,W_{j}}\,(\,T^{\,\ast}\,f\,) \,\right\|^{\,2}\; \;[\;\text{using (\ref{eq2.4})}\;]\hspace{2cm}\]
\[\hspace{2cm}\geq\; m^{\,2}\; A\; \left\|\,K^{\,\ast}\,T^{\,\ast}\,f\,\right\|^{\,2}\; \;[\;\text{since}\; \,\Lambda \;\text{is $K$-$g$-fusion frame}\;]\hspace{2cm}\]
\[=\, m^{\,2}\; A\; \left\|\,T^{\,\ast}\,K^{\,\ast}\,f\,\right\|^{\,2}\; \;[\;\text{because}\; \;K\,T \,=\, T\,K\;]\hspace{1cm}\]
\[\hspace{.3cm} \,\geq\, m^{\,2}\, A\, \left\|\,T^{\,-\, 1}\,\right\|^{\,-\, 2}\, \|\,K^{\,\ast}\,f\,\|^{\,2}\; \;[\;\text{using (\ref{eq2.5})}\;].\hspace{2cm}\]
On the other hand, for each \,$f \,\in\, H$, we have 
\[\sum\limits_{\,j \,\in\, J}\,v_{j}^{\,2}\, \left\|\,T_{j}\,\Lambda_{j}\,P_{\,W_{j}}\,T^{\,\ast} \,P_{\,T\,W_{j}}\,(\,f\,) \,\right\|^{\,2} \,=\, \sum\limits_{\,j \,\in\, J}\,v_{j}^{\,2}\, \left\|\,T_{j}\,\Lambda_{j}\,P_{\,W_{j}}\,(\,T^{\,\ast}\,f\,)\,\right\|^{\,2}\] 
\[\hspace{1.4cm}\leq\; \sum\limits_{\,j \,\in\, J}\, \left\|\,T_{j}\,\right\|^{\,2}\, v_{j}^{\,2}\; \left\|\,\Lambda_{j}\,P_{\,W_{j}}\,(\,T^{\,\ast}\,f\,) \,\right\|^{\,2}\]
\[\,\leq\, M^{\,2}\; \sum\limits_{\,j \,\in\, J}\,v_{j}^{\,2}\,\left\|\,\Lambda_{j}\,P_{\,W_{j}}\,(\,T^{\,\ast}\,f\,) \,\right\|^{\,2}\; \;[\;\text{using (\ref{eq2.4})}\;]\] 
\[\hspace{1cm}\leq\; M^{\,2}\; B\, \|\,T\,\|^{\,2}\; \|\,f\,\|^{\,2}\; \;[\;\text{since}\; \,\Lambda \;\text{is $K$-$g$-fusion frame}\;].\]
Thus, \,$\Gamma$\; is a \,$K$-$g$-fusion frame with bounds \,$m^{\,2}\, A\, \left\|\,T^{\,-\, 1}\,\right\|^{\,-\, 2}$\; and \,$M^{\,2}\; B\, \|\,T\,\|^{\,2}$.
\end{proof}

\begin{remark}
We further notice that the g-fusion frame operator \,$S_{\Gamma}$\, of \,$\Gamma$\; satisfies the followings
\begin{itemize}
\item[(I)]By (\ref{eq1.1}),
\[S_{\Gamma}\,f \,=\, \sum\limits_{\,j \,\in\, J}\,v_{j}^{\,2}\,P_{\,T\,W_{j}}\,\left(\,T_{j}\,\Lambda_{j}\,P_{\,W_{j}}\,T^{\,\ast}\,\right)^{\,\ast}\,\left(\,T_{j}\,\Lambda_{j}\,P_{\,W_{j}}\,T^{\,\ast}\,\right)\,P_{\,T\,W_{j}}\,(\,f\,)\]
\[\hspace{1.2cm}=\, \sum\limits_{\,j \,\in\, J}\, v_{j}^{\,2}\, \left(\,P_{\,W_{j}}\,T^{\,\ast}\, P_{\,T\,W_{j}}\,\right)^{\,\ast}\, \Lambda^{\,\ast}_{j}\; T^{\,\ast}_{j}\; T_{j}\; \Lambda_{j}\, \left(\,P_{\,W_{j}}\,T^{\,\ast}\,P_{\,T\,W_{j}}\,\right)\,(\,f\,)\]
\[\hspace{1.7cm}=\, \sum\limits_{\,j \,\in\, J}\, v_{j}^{\,2}\, \left(\,P_{\,W_{j}}\,T^{\,\ast}\,\right)^{\,\ast}\, \Lambda^{\,\ast}_{j}\; T^{\,\ast}_{j}\; T_{j}\; \Lambda_{j}\, P_{\,W_{j}}\,T^{\,\ast}\,(\,f\,)\; \;[\;\text{by Theorem (\ref{th1})}\;]\]
\begin{equation}\label{eq2.6}
=\, \sum\limits_{\,j \,\in\, J}\, v_{j}^{\,2}\; T\, P_{\,W_{j}}\, \Lambda^{\,\ast}_{j}\; T^{\,\ast}_{j}\; T_{j}\, \Lambda_{j}\, P_{\,W_{j}}\,T^{\,\ast}\,(\,f\,).\hspace{2.3cm}
\end{equation}
\item[(II)]
Moreover, if \,$K \,=\, I_{H}$, i\,.\,e, if \,$\Lambda$\; is a g-fusion frame then \,$\Gamma$\; is also g-fusion frame.\;Then \,$S_{\Gamma}$\, is invertible on \,$H$\, and by Theorem (\ref{th1.01}), we can write 
\begin{equation}\label{eq2.7}
\dfrac{1}{M^{\,2}\, B\, \|\,T\,\|^{\,2}}\; I_{H} \,\leq\, S^{\,-\, 1}_{\Gamma} \,\leq\, \dfrac{1}{m^{\,2}\, A\, \left\|\,T^{\,-\, 1}\,\right\|^{\,-\, 2}}\; I_{H}.
\end{equation}
\end{itemize}  
\end{remark}

\begin{remark}
Let us now denote \,$U \,=\, T^{\,\ast}\,S^{\,-\, 1}_{\Gamma}\,T $\; and for each \,$j \,\in\, J,\, L_{j} \,=\, T^{\,\ast}_{j}\,T_{j}$\; and \,$\Delta_{j} \,=\, L_{j}\,\Lambda_{j}\,P_{\,W_{j}}\,U $, where \,$T,\, T_{j},\, \Lambda \;\text{and}\; \;\Gamma$\; are all as in the Theorem (\ref{thm3}).\;Now it is easy to verify the following:
\begin{itemize}
\item[(I)]\hspace{.2cm}$U \,\in\, \mathcal{B}\,(\,H\,)$,
\item[(II)]\hspace{.2cm}for all \,$j \,\in\, J,\, L_{j} \,\in\, \mathcal{B}\,(\,H_{j}\,)$\, and \,$\Delta_{j} \,\in\, \mathcal{B}\,(\,H,\, H_{j}\,)$.
\item[(III)]\hspace{.2cm} \,$U$\; and \,$L_{j}$\; are self-adjoint and invertible.
\item[(IV)]\hspace{.2cm}From (\ref{eq2.7}), it can be obtained
\begin{equation}\label{eq2.8}
\|\,U\,\| \,\leq\, \|\,T^{\,\ast}\,\|\, \left\|\,S^{\,-\, 1}_{\Gamma}\,\right\|\, \|\,T\,\| \,\leq\, \dfrac{\|\,T\,\|^{\,2}}{m^{\,2}\, A\, \left\|\,T^{\,-\, 1}\,\right\|^{\,-\, 2}}\,.
\end{equation}
\item[(V)]\hspace{.2cm}For each \,$j \,\in\, J$, using (\ref{eq2.4}), 
\begin{equation}\label{eq2.9}
\left\|\,L_{j}\,\right\| \,=\,  \left\|\,T^{\,\ast}_{j}\,T_{j}\,\right\| \,=\, \left\|\,T_{j}\,\right\|^{\,2} \,\leq\, M^{\,2}.
\end{equation}  
\end{itemize}
\end{remark}

\begin{theorem}
Let \,$\Lambda$\, be a g-fusion frame for \,$H$\, with bounds \,$A$\, and \,$B$.\;Then \,$\Delta \,=\, \left\{\,\left(\,T\,W_{j},\; \Delta_{j},\; v_{j}\,\right)\,\right\}_{j \,\in\, J}$\; is a g-fusion frame for \,$H$.\;Furthermore,
\[f \,=\, \sum\limits_{\,j \,\in\, J}\, v_{j}^{\,2}\, P_{\,W_{j}}\, \Lambda^{\,\ast}_{j}\, \Delta_{j}\,P_{\,T\,W_{j}}\,(\,f\,) \,=\, \sum\limits_{\,j \,\in\, J}\, v_{j}^{\,2}\, P_{\,T\,W_{j}}\, \Delta^{\,\ast}_{j}\, \Lambda_{j}\, P_{\,W_{j}}\,(\,f\,)\;\; \;\forall\; f \,\in\, H.\] 
\end{theorem}
       
\begin{proof}
For each \,$f \,\in\, H$, we have
\[\sum\limits_{\,j \,\in\, J}\, v_{j}^{\,2}\, \left\|\,\Delta_{j}\,P_{\,T\,W_{j}}\,(\,f\,) \,\right\|^{\,2} \,=\, \sum\limits_{\,j \,\in\, J}\, v_{j}^{\,2}\, \left\|\,L_{j}\,\Lambda_{j}\,P_{\,W_{j}}\,U\,P_{\,T\,W_{j}}\,(\,f\,) \,\right\|^{\,2}\]
\[\hspace{4cm} =\, \sum\limits_{\,j \,\in\, J}\, v_{j}^{\,2}\, \left\|\,L_{j}\,\Lambda_{j}\,P_{\,W_{j}}\,(\,U\,f\,) \,\right\|^{\,2}\; \;[\;\text{by Theorem (\ref{th1})}\;]\]
\[\hspace{1.2cm}\leq\, \sum\limits_{\,j \,\in\, J}\, v_{j}^{\,2}\, \left\|\,L_{j}\,\right\|^{\,2}\, \left\|\,\Lambda_{j}\,P_{\,W_{j}}\,(\,U\,f\,)\,\right\|^{\,2}\]
\[\hspace{3cm}\,\leq\, M^{\,4}\, \sum\limits_{\,j \,\in\, J}\, v_{j}^{\,2}\, \left\|\,\Lambda_{j}\,P_{\,W_{j}}\,(\,U\,f\,)\,\right\|^{\,2}\; \;[\;\text{using (\ref{eq2.9})}\;]\]
\[\hspace{2.9cm} \,\leq\, B\, M^{\,4}\, \left\|\,U\,f\,\right\|^{\,2}\; \;[\;\text{since}\; \,\Lambda\; \;\text{is $g$-fusion frame}\;]\]
\[\hspace{2cm}\leq\, \dfrac{B\; M^{\,4}\; \|\,T\,\|^{\,4}} {m^{\,4}\; A^{\,2}\; \left\|\,T^{\,-\, 1}\,\right\|^{\,-\, 4}}\; \|\,f\,\|^{\,2}\; \;[\;\text{by (\ref{eq2.8})}\;].\]
Since for all \,$j \,\in\, J,\, L_{j}$\, is invertible so 
\[\sum\limits_{\,j \,\in\, J}\, v_{j}^{\,2}\, \left\|\,\Delta_{j}\,P_{\,T\,W_{j}}\,(\,f\,) \,\right\|^{\,2} \,=\, \sum\limits_{\,j \,\in\, J}\, v_{j}^{\,2}\, \left\|\,L_{j}\,\Lambda_{j}\,P_{\,W_{j}}\,U\,P_{\,T\,W_{j}}\,(\,f\,) \,\right\|^{\,2}\]
\[\hspace{4.7cm}\geq\, \sum\limits_{\,j \,\in\, J}\,\dfrac{1}{\left\|\,L_{j}^{\,-\, 1}\,\right\|^{\,2}}\;v_{j}^{\,2}\; \left\|\,\Lambda_{j}\,P_{\,W_{j}}\,(\,U\,f\,) \,\right\|^{\,2}\]
\[\geq\, m^{\,2}_{\,1}\; A\; \|\,U\,f\,\|^{\,2}\;  \;[\;\;\text{taking}\;\;m_{\,1} \,=\, \inf\limits_{j \,\in\, J}\,\dfrac{1}{\left\|\,L_{j}^{\,-\, 1}\,\right\|}\;]\]
\[ \geq\, \dfrac{m^{\,2}_{\,1}\; A}{\left\|\,U^{\,-\, 1}\,\right\|^{\,2}}\, \|\,f\,\|^{\,2}\; \;[\;\text{since}\; \,U\; \;\text{is invertible}\;].\hspace{.8cm}\]
Therefore, \,$\Delta$\; is a \,$g$-fusion frame for \,$H$.\;Furthermore, for each \,$f \,\in\, H$, we have 
\[\sum\limits_{\,j \,\in\, J}\, v_{j}^{\,2}\, P_{\,W_{j}}\, \Lambda^{\,\ast}_{j}\, \Delta_{j}\,P_{\,T\,W_{j}}\,(\,f\,) \,=\, \sum\limits_{\,j \,\in\, J}\, v_{j}^{\,2}\, P_{\,W_{j}}\, \Lambda^{\,\ast}_{j}\,\left(\,L_{j}\,\Lambda_{j}\,P_{\,W_{j}}\,U\,\right)\,P_{\,T\,W_{j}}\,(\,f\,)\]
\[=\, \sum\limits_{\,j \,\in\, J}\, v_{j}^{\,2}\, P_{\,W_{j}}\, \Lambda^{\,\ast}_{j}\; T^{\,\ast}_{j}\; T_{j}\,\Lambda_{j}\,\left(\,P_{\,W_{j}}\,U\,P_{\,T\,W_{j}}\,(\,f\,)\,\right)\]
\[\hspace{1.8cm}=\, \sum\limits_{\,j \,\in\, J}\, v_{j}^{\,2}\, P_{\,W_{j}}\, \Lambda^{\,\ast}_{j}\; T^{\,\ast}_{j}\; T_{j}\,\Lambda_{j}\,P_{\,W_{j}}\,U\,(\,f\,)\;[\;\text{by Theorem (\ref{th1})}\;]\]
\[=\, \sum\limits_{\,j \,\in\, J}\, v_{j}^{\,2}\, P_{\,W_{j}}\, \Lambda^{\,\ast}_{j}\; T^{\,\ast}_{j}\; T_{j}\,\Lambda_{j}\,P_{\,W_{j}}\,\left(\,T^{\,\ast}\,S^{\,-\, 1}_{\Gamma}\,T\,f\,\right)\hspace{.2cm}\]
\[\hspace{1.9cm}=\, T^{\,-\, 1}\, \left(\,\sum\limits_{\,j \,\in\, J}\, v_{j}^{\,2}\; T\, P_{\,W_{j}}\, \Lambda^{\,\ast}_{j}\; T^{\,\ast}_{j}\; T_{j}\,\Lambda_{j}\,P_{\,W_{j}}\,T^{\,\ast}\,\left(\,S^{\,-\, 1}_{\Gamma}\,T\,f\,\right)\,\right)\] 
\[=\, T^{\,-\, 1}\, S_{\Gamma}\,\left(\,S^{\,-\, 1}_{\Gamma}\; T\,f\,\,\right) \,=\, f\; \;[\;\text{using (\ref{eq2.6})}\;].\hspace{1cm}\]According to the preceding procedure, we also get 
\[f \,=\, \sum\limits_{\,j \,\in\, J}\, v_{j}^{\,2}\, P_{\,T\,W_{j}}\, \Delta^{\,\ast}_{j}\, \Lambda_{j}\, P_{\,W_{j}}\,(\,f\,)\;\; \;\forall\; f \,\in\, H.\]This completes the proof.           
\end{proof}

\section{Stability of dual \,$g$-fusion frame}

\smallskip\hspace{.6 cm} We know that if \,$\Lambda \,=\, \left\{\,\left(\,W_{j},\, \Lambda_{j},\, v_{j}\,\right)\,\right\}_{j \,\in\, J}$\; is a g-fusion frame for \,$H$\; with associated frame operator \,$S_{\Lambda}$\, then \,$\Lambda^{\circ} \,=\, \left\{\,\left(\,S_{\Lambda}^{\,-\, 1}\,W_{j},\; \Lambda_{j}\,P_{\,W_{j}}\,S_{\Lambda}^{\,-\, 1},\; v_{j}\,\right)\,\right\}_{j \,\in\, J}$\, is called as the canonical dual g-fusion frame of \,$\Lambda$.\;For each \,$f \,\in\, H$, the frame operator \,$S_{\Lambda^{\circ}}$\, of \,$\Lambda^{\circ}$\, is described by
\[S_{\Lambda^{\circ}}\,(\,f\,) \,=\, \sum\limits_{\,j \,\in\, J}\, v_{j}^{\,2}\, P_{\,S_{\Lambda}^{\,-\, 1}\,W_{j}}\,\left(\,\Lambda_{j}\,P_{\,W_{j}}\,S_{\Lambda}^{\,-\, 1}\,\right)^{\,\ast}\,\left(\,\Lambda_{j}\,P_{\,W_{j}}\,S_{\Lambda}^{\,-\, 1}\,\right)\,P_{\,S_{\Lambda}^{\,-\, 1}\,W_{j}}\,(\,f\,)\; \;[\;\text{by (\ref{eq1.1})}\;]\]
\[=\, \sum\limits_{\,j \,\in\, J}\,v_{j}^{\,2}\,P_{\,S_{\Lambda}^{\,-\, 1}\,W_{j}}\, S_{\Lambda}^{\,-\, 1}\, P_{\,W_{j}}\, \Lambda^{\,\ast}_{j}\,\Lambda_{j}\, \left(\,P_{\,W_{j}}\,S_{\Lambda}^{\,-\, 1}\,\,P_{\,S_{\Lambda}^{\,-\, 1}\,W_{j}}\,\right)\,(\,f\,)\hspace{.5cm}\]
\[\hspace{.5cm}=\, \sum\limits_{\,j \,\in\, J}\,v_{j}^{\,2}\,\left(\,P_{\,W_{j}}\,S_{\Lambda}^{\,-\, 1}\,P_{\,S_{\Lambda}^{\,-\, 1}\,W_{j}}\,\right)^{\,\ast}\, \Lambda^{\,\ast}_{j}\,\Lambda_{j}\, \left(\,P_{\,W_{j}}\,S_{\Lambda}^{\,-\, 1}\,\,P_{\,S_{\Lambda}^{\,-\, 1}\,W_{j}}\,\right)\,(\,f\,)\]
\[=\, \sum\limits_{\,j \,\in\, J}\,v_{j}^{\,2}\, \left(\,P_{\,W_{j}}\,S_{\Lambda}^{\,-\, 1}\,\right)^{\,\ast}\, \Lambda^{\,\ast}_{j}\,\Lambda_{j}\, P_{\,W_{j}}\,S_{\Lambda}^{\,-\, 1}\,(\,f\,)\; \;[\;\text{by Theorem (\ref{th1})}\;]\]
\[=\, \sum\limits_{\,j \,\in\, J}\,v_{j}^{\,2}\, S_{\Lambda}^{\,-\, 1}\,P_{\,W_{j}}\, \Lambda^{\,\ast}_{j}\,\Lambda_{j}\, P_{\,W_{j}}\,S_{\Lambda}^{\,-\, 1}\,(\,f\,)\hspace{4.4cm}\]
\[ \,=\, S_{\Lambda}^{\,-\, 1}\,\sum\limits_{\,j \,\in\, J}\,v_{j}^{\,2}\, P_{\,W_{j}}\, \Lambda^{\,\ast}_{j}\,\Lambda_{j}\, P_{\,W_{j}}\,S_{\Lambda}^{\,-\, 1}\,(\,f\,)\hspace{4.3cm}\]
\begin{equation}\label{eq3.1}
\,=\, S_{\Lambda}^{\,-\, 1}\,\left(\,S_{\Lambda}\,\left(\,S_{\Lambda}^{\,-\, 1}\,f\,\right)\,\right) \,=\, S_{\Lambda}^{\,-\, 1}\,(\,f\,).\hspace{4.7cm}
\end{equation}          

In this section, we shall study the stability of dual\;$g$-fusion frames and at the end, a necessary and sufficient condition for some \,$K \,\in\, \mathcal{B}\,(\,H\,)$, for some invertible operator \,$U \,\in\, \mathcal{B}\,(\,H\,)$, a quotient operator will be bounded if and only if \,$g$-fusion frame becomes \,$U\,K$-$g$-fusion frame.

\begin{theorem}\label{th4}
Let \,$\Lambda \,=\, \left\{\,\left(\,W_{j},\, \Lambda_{j},\, v_{j}\,\right)\,\right\}_{j \,\in\, J}$\; and \;$\Gamma \,=\, \left\{\,\left(\,V_{j},\, \Gamma_{\,j},\, v_{j}\,\right)\,\right\}_{j \,\in\, J}$\; be two g-fusion frames for \,$H$.\;If the condition
\[\sum\limits_{\,j \,\in\, J}\,v_{j}^{\,2}\,\left\|\,\left(\,\Lambda_{j}\,P_{\,W_{j}} \,-\, \Gamma_{j}\,P_{\,V_{j}}\,\right)\,(\,f\,)\,\right\|^{\,2} \,\leq\, D\, \|\,f\,\|^{\,2}\]
holds for each \,$f \,\in\, H$\, and for some \,$D \,>\, 0$\, then \,$\sum\limits_{\,j \,\in\, J}\,v_{j}\,\left(\,P_{\,W_{j}}\,\Lambda_{j}^{\,\ast} \,-\, P_{\,V_{j}}\,\Gamma_{j}^{\,\ast}\,\right)\,f_{j}$\, converges unconditionally for each \,$\{\,f_{\,j}\,\}_{j \,\in\, J} \;\in\; l^{\,2}\left(\,\left\{\,H_{j}\,\right\}_{ j \,\in\, J}\,\right)$. 
\end{theorem}

\begin{proof}
Let \,$I$\, be any finite subset of \,$J$.\;Then by Cauchy-Schwarz inequality for each \,$\{\,f_{\,j}\,\}_{j \,\in\, J} \;\in\; l^{\,2}\left(\,\left\{\,H_{j}\,\right\}_{ j \,\in\, J}\,\right)$, we have
\[\left\|\,\sum\limits_{\,j \,\in\, I}\,v_{j}\,\left(\,P_{\,W_{j}}\,\Lambda_{j}^{\,\ast} \,-\, P_{\,V_{j}}\,\Gamma_{j}^{\,\ast}\,\right)\,f_{j}\,\right\| \,=\, \sup\limits_{\|\,g\| \,=\, 1}\, \left|\,\left<\,\sum\limits_{\,j \,\in\, I}\,v_{j}\,\left(\,P_{\,W_{j}}\,\Lambda_{j}^{\,\ast} \,-\, P_{\,V_{j}}\,\Gamma_{j}^{\,\ast}\,\right)\,f_{j} \;,\; g\,\right>\,\right| \]
\[=\, \sup\limits_{\|\,g\| \,=\, 1}\,\left|\,\sum\limits_{\,j \,\in\, I}\,\left<\,f_{j} \;,\; v_{j}\,\left(\,\Lambda_{j}\,P_{\,W_{j}} \,-\, \Gamma_{j}\,P_{\,V_{j}}\,\right)\,(\,g\,)\,\right>\,\right|\hspace{1.3cm}\]
\[\hspace{2cm}\leq\, \left(\,\sum\limits_{\,j \,\in\, I}\,\left\|\,f_{j}\,\right\|^{\,2}\,\right)^{\dfrac{1}{2}}\, \sup\limits_{\|\,g\| \,=\, 1}\, \left(\,\sum\limits_{\,j \,\in\, I}\,v_{j}^{\,2}\,\left\|\,\left(\,\Lambda_{j}\,P_{\,W_{j}} \,-\, \Gamma_{j}\,P_{\,V_{j}}\,\right)\,(\,g\,)\,\right\|^{\,2}\,\right)^{\dfrac{1}{2}}\]
\[\leq\, D^{\dfrac{1}{2}}\; \left(\,\sum\limits_{\,j \,\in\, I}\,\left\|\,f_{j}\,\right\|^{\,2}\,\right)^{\dfrac{1}{2}} \,<\, \infty \hspace{4.6cm}\] and therefore \,$\sum\limits_{\,j \,\in\, J}\,v_{j}\,\left(\,P_{\,W_{j}}\,\Lambda_{j}^{\,\ast} \,-\, P_{\,V_{j}}\,\Gamma_{j}^{\,\ast}\,\right)\,f_{j}$\,  is unconditionally convergent in \,$H$.    
\end{proof}

\begin{theorem}
Let \,$\Lambda \,=\, \left\{\,\left(\,W_{j},\, \Lambda_{j},\, v_{j}\,\right)\,\right\}_{j \,\in\, J}$\; and \;$\Gamma \,=\, \left\{\,\left(\,V_{j},\, \Gamma_{j},\, v_{j}\,\right)\,\right\}_{j \,\in\, J}$\; be two g-fusion frames for \,$H$\; with frame bounds \,$(\,A_{\,1},\, B_{\,1}\,)$\; and \,$(\,A_{\,2},\, B_{\,2}\,)$, respectively.\;Take \,$\Lambda^{\circ} \,=\, \left\{\,\left(\,W_{j}^{\,\circ},\, \Lambda_{j}^{\,\circ},\, v_{j}\,\right)\,\right\}_{j \,\in\, J}$\; and \,$\Gamma^{\circ} \,=\, \left\{\,\left(\,V_{j}^{\,\circ},\, \Gamma_{j}^{\,\circ},\, v_{j}\,\right)\,\right\}_{j \,\in\, J}$\; be the corresponding canonical dual g-fusion frames for \,$\Lambda$\; and \,$\Gamma$, respectively.\;Then the following statements hold:
\begin{itemize}
\item[(I)]\hspace{.1cm} If the condition 
\[\sum\limits_{\,j \,\in\, J}\,v_{j}^{\,2}\,\left\|\,\left(\,\Lambda_{j}\,P_{\,W_{j}} \,-\, \Gamma_{j}\,P_{\,V_{j}}\,\right)\,f\,\right\|^{\,2} \,\leq\, D\, \|\,f\,\|^{\,2}\]holds for each \,$f \,\in\, H$\, and for some \,$D \,>\, 0$\, then for all \,$f \,\in\, H$,
\[\sum\limits_{\,j \,\in\, J}\,v_{j}^{\,2}\,\left\|\,\left(\,\Lambda_{j}^{\,\circ}\,P_{\,W_{j}^{\,\circ}} \,-\, \Gamma_{j}^{\,\circ}\,P_{\,V_{j}^{\,\circ}}\,\right)\,f\,\right\|^{\,2} \,\leq\, D\, \left(\,\dfrac{A_{\,1} \,+\, B_{\,1} \,+\, B_{\,1}^{\dfrac{1}{2}}\,B_{\,2}^{\dfrac{1}{2}}}{A_{\,1}\,A_{\,2}}\,\right)^{\,2}\,\|\,f\,\|^{\,2}.\]
\item[(II)]\hspace{.1cm} If the condition 
\[\left|\,\sum\limits_{\,j \,\in\, J}\,v_{j}^{\,2}\, \left\|\,\Lambda_{j}\,P_{\,W_{j}}\,(\,f\,) \,\right\|^{\,2} \,-\, \sum\limits_{\,j \,\in\, J}\,v_{j}^{\,2}\, \left\|\,\Gamma_{j}\,P_{\,V_{j}}\,(\,f\,) \,\right\|^{\,2}\,\right| \,\leq\, D\; \|\,f\,\|^{\,2}\]holds for each \,$f \,\in\, H$\, and for some \,$D \,>\, 0$\, then for all \,$f \,\in\, H$,
\[\left|\,\sum\limits_{\,j \,\in\, J}\,v_{j}^{\,2}\, \left\|\,\Lambda_{j}^{\,\circ}\,P_{\,W_{j}^{\,\circ}}\,(\,f\,) \,\right\|^{\,2} \,-\, \sum\limits_{\,j \,\in\, J}\,v_{j}^{\,2}\, \left\|\,\Gamma_{j}^{\,\circ}\,P_{\,V_{j}^{\,\circ}}\,(\,f\,) \,\right\|^{\,2}\,\right| \;\leq\; \dfrac{D}{A_{\,1}\,A_{\,2}}\;\|\,f\,\|^{\,2}.\]
\end{itemize} 
\end{theorem}

\begin{proof}(I) Let \,$S_{\Lambda}$\; and \,$S_{\Gamma}$\; be the corresponding \,$g$-fusion frame operators for \,$\Lambda$\; and \,$\Gamma$, then for each \,$f \,\in\, H$, we have  
\[S_{\,\Lambda}\,f \,=\, \sum\limits_{\,j \,\in\, J}\,v_{j}^{\,2}\,P_{\,W_{j}}\,\Lambda_{j}^{\,\ast}\,\Lambda_{j}\,P_{\,W_{j}}\,(\,f\,),\;\; S_{\,\Gamma}\,f \,=\, \sum\limits_{\,j \,\in\, J}\,v_{j}^{\,2}\,P_{\,V_{j}}\,\Gamma_{j}^{\,\ast}\,\Gamma_{j}\,P_{\,V_{j}}\,(\,f\,), \;\&\] 
\begin{equation}\label{eq3.11}
B^{\,-\, 1}_{\,1}\,I_{H} \,\leq\, S^{\,-\, 1}_{\Lambda} \,\leq\, A^{\,-\, 1}_{\,1}\,I_{H},\;  \;B^{\,-\, 1}_{\,2}\,I_{H} \,\leq\, S^{\,-\, 1}_{\Gamma} \,\leq\, A^{\,-\, 1}_{\,2}\,I_{H}.
\end{equation} 
Since \,$\Lambda^{\circ}$\; and \,$\Gamma^{\circ}$\; are canonical dual\;$g$-fusion frames of \,$\Lambda$\; and \,$\Gamma$, so 
\[W_{j}^{\,\circ} \,=\, S_{\Lambda}^{\,-\, 1}\,W_{j},\;\Lambda_{j}^{\,\circ} \,=\, \Lambda_{j}\,P_{\,W_{j}}\,S_{\Lambda}^{\,-\, 1}\; \;\text{and}\; \;V_{j}^{\,\circ} \,=\, S_{\Gamma}^{\,-\, 1}\,V_{j},\; \Gamma_{j}^{\,\circ} \,=\, \Gamma_{j}\,P_{\,V_{j}}\,S_{\Gamma}^{\,-\, 1}.\]
Then for any \,$f \,\in\, H$, using Theorem (\ref{th1.1}) and the proof of Theorem (\ref{th4}), we get
\[\left\|\,S_{\Lambda}\,f \,-\, S_{\Gamma}\,f\,\right\| \,=\, \left\|\,\sum\limits_{\,j \,\in\, J}\,v_{j}^{\,2}\,\left(\,P_{\,W_{j}}\,\Lambda_{j}^{\,\ast}\,\Lambda_{j}\,P_{\,W_{j}}\,(\,f\,) \,-\, P_{\,V_{j}}\,\Gamma_{j}^{\,\ast}\,\Gamma_{j}\,P_{\,V_{j}}\,(\,f\,)\,\right)\,\right\|\]
\[\,=\, \left\|\,\sum\limits_{\,j \,\in\, J}\,v_{j}^{\,2}\,\left(\,P_{\,W_{j}}\,\Lambda_{j}^{\,\ast}\,\Lambda_{j}\,P_{\,W_{j}} \,-\, P_{\,W_{j}}\,\Lambda_{j}^{\,\ast}\,\Gamma_{j}\,P_{\,V_{j}} \,+\, P_{\,W_{j}}\,\Lambda_{j}^{\,\ast}\,\Gamma_{j}\,P_{\,V_{j}} \,-\, P_{\,V_{j}}\,\Gamma_{j}^{\,\ast}\,\Gamma_{j}\,P_{\,V_{j}}\,\right)\,(\,f\,)\,\right\|\]
\[\leq\, \left\|\,\sum\limits_{\,j \,\in\, J}\,v_{j}\;P_{\,W_{j}}\,\Lambda_{j}^{\,\ast}\;v_{j}\, \left(\,\Lambda_{j}\,P_{\,W_{j}} \,-\, \Gamma_{j}\,P_{\,V_{j}}\,\right)\,(\,f\,)\,\right\| \,+\, \left\|\,\sum\limits_{\,j \,\in\, J}\,v_{j}\,\left(\,P_{\,W_{j}}\,\Lambda_{j}^{\,\ast} \,-\, P_{\,V_{j}}\,\Gamma_{j}^{\,\ast}\,\right)\,v_{j}\,\Gamma_{j}\,P_{\,V_{j}}\,(\,f\,)\,\right\|\]
\[\leq\, B_{\,1}^{\dfrac{1}{2}}\,\left(\,\sum\limits_{\,j \,\in\, J}\,v_{j}^{\,2}\left\|\,\left(\,\Lambda_{j}\,P_{\,W_{j}} \,-\, \Gamma_{j}\,P_{\,V_{j}}\,\right)\,(\,f\,)\,\right\|^{\,2}\,\right)^{\dfrac{1}{2}} \,+\; D^{\dfrac{1}{2}}\,\left(\,\sum\limits_{\,j \,\in\, J}\,v_{j}^{\,2}\,\left\|\,\Gamma_{j}\,P_{\,V_{j}}\,(\,f\,)\,\right\|^{\,2}\,\right)^{\dfrac{1}{2}}\]
\[\leq\, B_{\,1}^{\dfrac{1}{2}}\; D^{\dfrac{1}{2}}\; \|\,f\,\| \,+\,  D^{\dfrac{1}{2}}\; B_{\,2}^{\dfrac{1}{2}}\; \|\,f\,\| \;=\; D^{\dfrac{1}{2}}\; \left(\,B_{\,1}^{\dfrac{1}{2}} \,+\, B_{\,2}^{\dfrac{1}{2}}\,\right)\; \|\,f\,\|\,.\hspace{3cm}\]
Therefore,
\[\hspace{2cm}\left\|\,S_{\Lambda} \,-\, S_{\Gamma}\,\right\| \,=\, \sup\limits_{\|\,f\,\| \,=\, 1}\,\left\|\,S_{\Lambda}\,f \,-\, S_{\Gamma}\,f\,\right\| \;\leq\; D^{\dfrac{1}{2}}\; \left(\,B_{\,1}^{\dfrac{1}{2}} \,+\, B_{\,2}^{\dfrac{1}{2}}\,\right).\] On the other hand, 
\[\left\|\,S_{\Lambda}^{\,-\, 1} \,-\, S_{\Gamma}^{\,-\, 1}\,\right\| \,=\, \left\|\,S_{\Lambda}^{\,-\, 1}\;(\,S_{\Lambda} \,-\, S_{\Gamma}\,)\;S_{\Gamma}^{\,-\, 1}\,\right\|\]
\[\hspace{3.7cm}\,\leq\, \left\|\,S_{\Lambda}^{\,-\, 1}\,\right\|\; \|\,S_{\Lambda} \,-\, S_{\Gamma}\,\|\; \left\|\,S_{\Gamma}^{\,-\, 1}\,\right\|\]
\[\hspace{4.8cm} \,\leq\; \dfrac{D^{\,\dfrac{1}{2}}}{A_{\,1}\,A_{\,2}}\;\left(\,B_{\,1}^{\dfrac{1}{2}} \,+\, B_{\,2}^{\dfrac{1}{2}}\,\right)\;[\;\text{by (\ref{eq3.11})}\;].\] Since \,$\Lambda$\; is a \,$g$-fusion frame, for \,$f \,\in\, H$, we have 
\[\sum\limits_{\,j \,\in\, J}\,v_{j}^{\,2}\, \left\|\,\Lambda_{j}\,P_{\,W_{j}}\,\left(\,S_{\Lambda}^{\,-\, 1} \,-\, S_{\Gamma}^{\,-\, 1}\,\right)\,f \,\right\|^{\,2} \,\leq\, B_{\,1}\; \left\|\,\left(\,S_{\Lambda}^{\,-\, 1} \,-\, S_{\Gamma}^{\,-\, 1}\,\right)\,f\,\right\|^{\,2}\]
\begin{equation}\label{eq3.2}
\hspace{3cm}\leq\; \dfrac{B_{\,1}}{A_{\,1}^{\,2}\,A_{\,2}^{\,2}}\;D \;\left(\,B_{\,1}^{\dfrac{1}{2}} \,+\, B_{\,2}^{\dfrac{1}{2}}\,\right)^{\,2}\,\|\,f\,\|^{\,2}.
\end{equation}
Also, by given condition, we obtain
\begin{equation}\label{eq3.3}
\sum\limits_{\,j \,\in\, J}\,v_{j}^{\,2}\left\|\,\left(\,\Lambda_{j}\,P_{\,W_{j}} \,-\, \Gamma_{j}\,P_{\,V_{j}}\,\right)\,S_{\Gamma}^{\,-\, 1}\,f\,\right\|^{\,2} \,\leq\, D\, \left\|\,S_{\Gamma}^{\,-\, 1}\,f\,\right\|^{\,2} \,\leq\, \dfrac{D}{A_{\,2}^{\,2}}\, \|\,f\,\|^{\,2}.
\end{equation}
Now, by Minkowski inequality, for each \,$f \,\in\, H$, we have
\[\sum\limits_{\,j \,\in\, J}\,v_{j}^{\,2}\,\left\|\,\left(\,\Lambda_{j}^{\,\circ}\,P_{\,W_{j}^{\,\circ}} \,-\, \Gamma_{j}^{\,\circ}\,P_{\,V_{j}^{\,\circ}}\,\right)\,(\,f\,)\,\right\|^{\,2}\]
\[=\; \sum\limits_{\,j \,\in\, J}\,v_{j}^{\,2}\,\left\|\,\left(\,\Lambda_{j}\,P_{\,W_{j}}\,S_{\Lambda}^{\,-\, 1}\;P_{\,S_{\Lambda}^{\,-\, 1}\,W_{j}} \,-\, \Gamma_{j}\,P_{\,V_{j}}\,S_{\Gamma}^{\,-\, 1}\,P_{\,S_{\Gamma}^{\,-\, 1}\,V_{j}}\,\right)\,(\,f\,)\,\right\|^{\,2}\hspace{2.5cm}\]
\[=\; \sum\limits_{\,j \,\in\, J}\,v_{j}^{\,2}\,\left\|\,\left(\,\Lambda_{j}\,P_{\,W_{j}}\,S_{\Lambda}^{\,-\, 1} \,-\, \Gamma_{j}\,P_{\,V_{j}}\,S_{\Gamma}^{\,-\, 1}\,\right)\,(\,f\,)\,\right\|^{\,2}\; \;[\;\text{by Theorem (\ref{th1})}\;] \hspace{4cm}\]
\[=\; \sum\limits_{\,j \,\in\, J}\,\left\|\,v_{j}\,\Lambda_{j}\,P_{\,W_{j}}\,\left(\,S_{\Lambda}^{\,-\, 1} \,-\, S_{\Gamma}^{\,-\, 1}\,\right)\,(\,f\,) \,+\, v_{j}\,\left(\,\Lambda_{j}\,P_{\,W_{j}} \,-\, \Gamma_{j}\,P_{\,V_{j}}\,\right)\, S_{\Gamma}^{\,-\, 1}\,(\,f\,)\,\right\|^{\,2} \hspace{4cm}\]
\[\leq\, \left(\,\left(\,\sum\limits_{\,j \,\in\, J}\,v_{j}^{\,2}\, \left\|\,\Lambda_{j}\,P_{\,W_{j}}\,\left(\,S_{\Lambda}^{\,-\, 1} \,-\, S_{\Gamma}^{\,-\, 1}\,\right)\,f\,\right\|^{\,2}\,\right)^{\,\dfrac{1}{2}} \,+\, \left(\,\sum\limits_{\,j \,\in\, J}\,v_{j}^{\,2}\left\|\,\left(\,\Lambda_{j}\,P_{\,W_{j}} \,-\, \Gamma_{j}\,P_{\,V_{j}}\,\right)\,S_{\Gamma}^{\,-\, 1}\,f\,\right\|^{\,2}\,\right)^{\,\dfrac{1}{2}}\,\right)^{\,2}\]
\[\leq\; \left(\,\dfrac{B_{\,1}^{\,\dfrac{1}{2}}}{A_{\,1}\,A_{\,2}}\; D^{\,\dfrac{1}{2}}\; \left(\,B_{\,1}^{\dfrac{1}{2}} \,+\, B_{\,2}^{\dfrac{1}{2}}\,\right) \;+\; \dfrac{D^{\,\dfrac{1}{2}}}{A_{\,2}}\,\right)^{\,2}\,\|\,f\,\|^{\,2}\; \;[\;\text{using (\ref{eq3.2}) and (\ref{eq3.3})}\;]\hspace{2cm}\]
\[=\; D\; \left(\,\dfrac{B_{\,1}^{\,\dfrac{1}{2}}}{A_{\,1}\,A_{\,2}}\; \left(\,B_{\,1}^{\dfrac{1}{2}} \,+\, B_{\,2}^{\dfrac{1}{2}}\,\right) \;+\; \dfrac{1}{A_{\,2}}\,\right)^{\,2}\,\|\,f\,\|^{\,2}\;=\; D\; \left(\,\dfrac{A_{\,1} \,+\, B_{\,1} \,+\, B_{\,1}^{\,\dfrac{1}{2}}\,B_{\,2}^{\dfrac{1}{2}}}{A_{\,1}\,A_{\,2}}\,\right)^{\,2}\,\|\,f\,\|^{\,2}.\]This completes the proof of (I).\\ 

Proof of (II).
Since \,$S_{\Lambda} \,-\, S_{\Gamma}$\, is self-adjoint so 
\[ \left\|\,S_{\Lambda} \,-\, S_{\Gamma}\,\right\| \,=\, \sup\limits_{\|\,f\,\| \,=\, 1}\,\left|\,\left<\,(\,S_{\Lambda} \,-\, S_{\Gamma}\,)\,f \;,\; f\,\right>\,\right| \,=\, \sup\limits_{\|\,f\,\| \,=\, 1}\,\left|\,\left<\,S_{\Lambda}\,f \;,\; f\,\right> \,-\, \left<\,S_{\Gamma}\,f \;,\; f \,\right>\,\right|\]
\[\hspace{2.3cm}=\; \sup\limits_{\|\,f\,\| \,=\, 1}\,\left|\,\sum\limits_{\,j \,\in\, I}\,v_{j}^{\,2}\, \left\|\,\Lambda_{j}\,P_{\,W_{j}}\,(\,f\,) \,\right\|^{\,2} \,-\, \sum\limits_{\,j \,\in\, I}\,v_{j}^{\,2}\, \left\|\,\Gamma_{j}\,P_{\,V_{j}}\,(\,f\,) \,\right\|^{\,2}\,\right| \,\leq\, D.\]
Therefore,
\[\left\|\,S_{\Lambda}^{\,-\, 1} \,-\, S_{\Gamma}^{\,-\, 1}\,\right\| \,\leq\, \left\|\,S_{\Lambda}^{\,-\, 1}\,\right\|\; \|\,S_{\Lambda} \,-\, S_{\Gamma}\,\|\; \left\|\,S_{\Gamma}^{\,-\, 1}\,\right\|\]
\begin{equation}\label{eq3.4}
\hspace{1.8cm}\leq\; \dfrac{1}{A_{\,1}}\; D\; \dfrac{1}{A_{\,2}} \;=\; \dfrac{D}{A_{\,1}\,A_{\,2}}\,.
\end{equation}
Now, for each \,$f \,\in\, H$, we have
\[\sum\limits_{\,j \,\in\, J}\,v_{j}^{\,2}\,\left\|\,\Lambda_{j}^{\,\circ}\,P_{\,W_{j}^{\,\circ}}\,(\,f\,)\,\right\|^{\,2} \,=\, \sum\limits_{\,j \,\in\, J}\,v_{j}^{\,2}\,\left\|\,\Lambda_{j}\,P_{\,W_{j}}\,S_{\Lambda}^{\,-\, 1}\;P_{\,S_{\Lambda}^{\,-\, 1}\,W_{j}}\,(\,f\,)\,\right\|^{\,2}\]
\[\hspace{5cm}\,=\,  \sum\limits_{\,j \,\in\, J}\,v_{j}^{\,2}\,\left\|\,\Lambda_{j}\,P_{\,W_{j}}\,\left(\,S_{\Lambda}^{\,-\, 1}\,f\,\right)\,\right\|^{\,2}\; \;[\;\text{by Theorem (\ref{th1})}\;]\]
\[=\; \sum\limits_{\,j \,\in\, J}\,\left<\,v_{j}^{\,2}\;\Lambda_{j}\,P_{\,W_{j}}\,\left(\,S_{\Lambda}^{\,-\, 1}\,f\,\right) \;,\; \Lambda_{j}\,P_{\,W_{j}}\,\left(\,S_{\Lambda}^{\,-\, 1}\,f\,\right)\,\right>\] 
\[=\; \left<\,\sum\limits_{\,j \,\in\, J}\,v_{j}^{\,2}\,P_{\,W_{j}}\,\Lambda_{j}^{\,\ast}\,\Lambda_{j}\,P_{\,W_{j}}\,\left(\,S_{\Lambda}^{\,-\, 1}\,f\,\right) \;,\; S_{\Lambda}^{\,-\, 1}\,f\,\right>\hspace{.3cm}\]
\[=\; \left<\,S_{\Lambda}\,\left(\,S_{\Lambda}^{\,-\, 1}\,f\,\right) \;,\; S_{\Lambda}^{\,-\, 1}\,f\,\right> \,=\, \left<\,f \;,\; S_{\Lambda}^{\,-\, 1}\,f\,\right>.\hspace{.7cm}\] Similarly it can be shown that
\[\sum\limits_{\,j \,\in\, J}\,v_{j}^{\,2}\,\left\|\,\Gamma_{j}^{\,\circ}\,P_{\,V_{j}^{\,\circ}}\,(\,f\,)\,\right\|^{\,2} \,=\, \left<\,f \;,\; S_{\Gamma}^{\,-\, 1}\,f\,\right>\;\; \;\forall\; f \,\in\, H.\]
Then for each \,$f \,\in\, H$, we have
\[\left|\,\sum\limits_{\,j \,\in\, J}\,v_{j}^{\,2}\, \left\|\,\Lambda_{j}^{\,\circ}\,P_{\,W_{j}^{\,\circ}}\,(\,f\,) \,\right\|^{\,2} \,-\, \sum\limits_{\,j \,\in\, J}\,v_{j}^{\,2}\, \left\|\,\Gamma_{j}^{\,\circ}\,P_{\,V_{j}^{\,\circ}}\,(\,f\,) \,\right\|^{\,2}\,\right|\]
\[\hspace{1.5cm}=\; \left|\,\left<\,f \;,\; S_{\Lambda}^{\,-\, 1}\,f\,\right> \,-\, \left<\,f \;,\; S_{\Gamma}^{\,-\, 1}\,f\,\right>\,\right| \,=\, \left|\,\left<\,f \;,\; \left(\,S_{\Lambda}^{\,-\, 1} \,-\, S_{\Gamma}^{\,-\, 1}\,\right)\,f\,\right>\,\right|\]
\[\leq\; \left\|\,S_{\Lambda}^{\,-\, 1} \,-\, S_{\Gamma}^{\,-\, 1}\,\right\|\; \|\,f\,\|^{\,2} \;\leq\; \dfrac{D}{A_{\,1}\,A_{\,2}}\; \|\,f\,\|^{\,2}\; \;[\;\text{by (\ref{eq3.4})}\;].\]
This completes the proof.
\end{proof}

\begin{remark}
Another representation of the statement (II) is given by,\\ if the condition 
\[\left\|\,\sum\limits_{\,j \,\in\, J}\, v_{j}^{\,2}\,\left(\,P_{\,W_{j}}\,\Lambda_{j}^{\,\ast}\;\Lambda_{j}\,P_{\,W_{j}}\,(\,f\,) \;-\; P_{\,V_{j}}\,\Gamma_{j}^{\,\ast}\;\Gamma_{j}\,P_{\,V_{j}}\,(\,f\,)\,\right)\,\right\| \,\leq\, D\; \|\,f\,\|\]holds for each \,$f \,\in\, H$\, and for some \,$D \,>\, 0$\, then for all \,$f \,\in\, H$,
\[\left\|\,\sum\limits_{\,j \,\in\, J}\, v_{j}^{\,2}\,\left(\,P_{\,W^{\,\circ}_{j}}\,(\,\Lambda_{j}^{\,\circ}\,)^{\,\ast}\;\Lambda^{\,\circ}_{j}\,P_{\,W^{\,\circ}_{j}}\,(\,f\,) \;-\; P_{\,V^{\,\circ}_{j}}\,(\,\Gamma_{j}^{\,\circ}\,)^{\,\ast}\;\Gamma^{\,\circ}_{j}\,P_{\,V^{\,\circ}_{j}}\,(\,f\,)\,\right)\,\right\| \,\leq\, \dfrac{D}{A_{\,1}\,A_{\,2}}\;\|\,f\,\|\,.\]

\begin{proof}
 In this case, we also find that  
\[ \left\|\,S_{\Lambda} \,-\, S_{\Gamma}\,\right\| \,=\, \sup\limits_{\|\,f\,\| \,=\, 1}\,\left\|\,S_{\Lambda}\,f \,-\, S_{\Gamma}\,f\,\right\|\]
\[=\, \sup\limits_{\|\,f\,\| \,=\, 1}\,\left\|\,\sum\limits_{\,j \,\in\, J}\, v_{j}^{\,2}\,\left(\,P_{\,W_{j}}\,\Lambda_{j}^{\,\ast}\;\Lambda_{j}\,P_{\,W_{j}}\,(\,f\,) \;-\; P_{\,V_{j}}\,\Gamma_{j}^{\,\ast}\;\Gamma_{j}\,P_{\,V_{j}}\,(\,f\,)\,\right)\,\right\|\]
\[ \,\leq\; \sup\limits_{\|\,f\,\| \,=\, 1}\,D\, \|\,f\,\| \,=\, D.\hspace{7.1cm}\] 
Then for each \,$f \,\in\, H$,
\[\left\|\,\sum\limits_{\,j \,\in\, I}\, v_{j}^{\,2}\,\left(\,P_{\,W^{\,\circ}_{j}}\,(\,\Lambda_{j}^{\,\circ}\,)^{\,\ast}\;\Lambda^{\,\circ}_{j}\,P_{\,W^{\,\circ}_{j}}\,(\,f\,) \;-\; P_{\,V^{\,\circ}_{j}}\,(\,\Gamma_{j}^{\,\circ}\,)^{\,\ast}\;\Gamma^{\,\circ}_{j}\,P_{\,V^{\,\circ}_{j}}\,(\,f\,)\,\right)\,\right\|\]
\[=\, \left\|\,S_{\Lambda^{\circ}}\,(\,f\,) \,-\, S_{\Gamma^{\circ}}\,(\,f\,)\,\right\| \,=\, \left\|\,S_{\Lambda}^{\,-\, 1}\,f \,-\, S_{\Gamma}^{\,-\, 1}\,f\,\right\|\; \;[\;\text{using (\ref{eq3.1})}\;]\]
\[\leq\, \left\|\,S_{\Lambda}^{\,-\, 1}\, \,-\, S_{\Gamma}^{\,-\, 1}\,\right\|\, \|\,f\,\| \,\leq\, \dfrac{D}{A_{\,1}\,A_{\,2}}\; \|\,f\,\|\; \;[\;\text{using (\ref{eq3.4})}\;].\hspace{1cm}\]
\end{proof}   
\end{remark}

\begin{theorem}
Let \,$K \,\in\, \mathcal{B}\,(\,H\,)$\; and \,$\Lambda \,=\, \left\{\,\left(\,W_{j},\, \Lambda_{j},\, v_{j}\,\right)\,\right\}_{j \,\in\, J}$\; be a K-g-fusion frame for \,$H$\; with frame operator \;$S_{\Lambda}$.\;Let \,$U \,\in\, \mathcal{B}\,(\,H\,)$\; be an invertible operator on \,$H$.\;Then the following statements are equivalent:
\begin{itemize}
\item[(I)]\hspace{.2cm} $\Gamma \,=\, \left\{\,\left(\,U\,W_{\,j},\, \Lambda_{\,j}\,P_{\,W_{\,j}}\,U^{\,\ast},\, v_{\,j}\,\right)\,\right\}_{j \,\in\, J}$\; is a U\,K-g-fusion frame. 
\item[(II)]\hspace{.2cm} The quotient operator \,$\left[\,\left(\,U\,K\,\right)^{\,\ast} \;/\; S_{\Lambda}^{\,\dfrac{1}{2}}\;U^{\,\ast} \,\right]$\; is bounded.
\item[(III)]\hspace{.2cm} The quotient operator \,$\left[\,\left(\,U\,K\,\right)^{\,\ast} \;/\; \left(\,U\,S_{\Lambda}\,U^{\,\ast}\,\right)^{\,\dfrac{1}{2}}\,\right]$\; is bounded.
\end{itemize}
\end{theorem}

\begin{proof}$(I) \,\Rightarrow\, (II)$\, Since \,$\Gamma$\; is a \,$U\,K$-$g$-fusion frame, \,$\exists\, \,A,\, B \,>\, 0$\, such that for all \,$f \,\in\, H$, we have
\begin{equation}\label{eq3.5} 
A\, \left\|\,\left(\,U\,K\,\right)^{\,\ast}\,f\,\right\|^{\,2} \,\leq\, \sum\limits_{\,j \,\in\, J}\,v_{j}^{\,2}\, \left\|\,\Lambda_{j}\,P_{\,W_{j}}\,U^{\,\ast}\,P_{\,U\,W_{j}}\,(\,f\,) \,\right\|^{\,2} \,\leq\, B\, \|\,f\,\|^{\,2}.
\end{equation}
By Theorem (\ref{th1}), for each \,$f \,\in\, H$, we obtain
\[\sum\limits_{\,j \,\in\, J}\,v_{j}^{\,2}\, \left\|\,\Lambda_{j}\,P_{\,W_{j}}\,U^{\,\ast}\,P_{\,U\,W_{j}}\,(\,f\,) \,\right\|^{\,2} \,=\, \sum\limits_{\,j \,\in\, J}\,v_{j}^{\,2}\, \left\|\,\Lambda_{j}\,P_{\,W_{j}}\,(\,U^{\,\ast}\,f\,) \,\right\|^{\,2}\]
\[\hspace{3cm}=\; \left<\,S_{\Lambda}\,\left(\,U^{\,\ast}\,f\,\right) \;,\; U^{\,\ast}\,f\,\right> \,=\, \left\|\,S_{\Lambda}^{\,\dfrac{1}{2}}\;\left(\,U^{\,\ast}\,f\,\right)\,\right\|^{\,2}\]and therefore from (\ref{eq3.5}),
\begin{equation}\label{eq3.6}
 A\; \left\|\,\left(\,U\,K\,\right)^{\,\ast}\;f\,\right\|^{\,2} \;\leq\; \left\|\,S_{\Lambda}^{\,\dfrac{1}{2}}\;\left(\,U^{\,\ast}\,f\,\right)\,\right\|^{\,2}.
\end{equation}
Let us define the operator \,$T \,:\, \mathcal{R}\,\left(\,S_{\Lambda}^{\,\dfrac{1}{2}}\;U^{\,\ast}\,\right) \,\to\, \mathcal{R}\,\left(\,(\,U\,K\,)^{\,\ast}\,\right)$\, by 
\[T\,\left(\,S_{\Lambda}^{\,\dfrac{1}{2}}\;U^{\,\ast}\,f\,\right) \,=\, (\,U\,K\,)^{\,\ast}\,f\;\; \;\forall\; f \,\in\, H.\]
Then it can be easily verify that \,$T$\, is a linear operator and \,$\mathcal{N}\,\left(\,S_{\Lambda}^{\,\dfrac{1}{2}}\;U^{\,\ast}\,\right) \,\subset\, \mathcal{N}\,\left(\,\left(\,U\,K\,\right)^{\,\ast}\,\right)$.\;Thus \,$T$\, is well-defined quotient operator.\;Also for each \,$f \,\in\, H$,
\[\left\|\,T\,\left(\,S_{\Lambda}^{\,\dfrac{1}{2}}\;U^{\,\ast}\,f\,\right)\,\right\| \,=\, \left\|\,(\,U\,K\,)^{\,\ast}\,f\,\right\| \,\leq\, \dfrac{1}{\sqrt{A}} \;\left\|\,S_{\Lambda}^{\,\dfrac{1}{2}}\;\left(\,U^{\,\ast}\,f\,\right)\,\right\|\; \;[\;\text{using (\ref{eq3.6})}\;]\]and hence \,$T$\; is bounded.

$(II) \,\Rightarrow\, (III)$\; Suppose the quotient operator \,$\left[\,\left(\,U\,K\,\right)^{\,\ast} \;/\; S_{\Lambda}^{\,\dfrac{1}{2}}\;U^{\,\ast} \,\right]$\; is bounded. Then for each \,$f \,\in\, H, \,\exists\; B \,>\, 0$\; such that 
\[\left\|\,\left(\,U\,K\,\right)^{\,\ast}\,f\,\right\|^{\,2} \,\leq\, B\, \left\|\,S_{\Lambda}^{\,\dfrac{1}{2}}\,\left(\,U^{\,\ast}\,f\,\right)\,\right\|^{\,2} \,=\, B\, \left<\,S_{\Lambda}\,\left(\,U^{\,\ast}\,f\,\right) \,,\, U^{\,\ast}\,f\,\right> \,=\, B\, \left<\,U\,S_{\Lambda}\,U^{\,\ast}\,f \,,\, f\,\right>\]
\[\hspace{5cm}\,=\, B\; \left\|\,\left(\,U\,S_{\Lambda}\,U^{\,\ast}\,\right)^{\dfrac{1}{2}}\,f\,\right\|^{\,2}\; \;[\;\text{since}\; \,U\,S_{\Lambda}\,U^{\,\ast}\; \;\text{is self-adjoint}\;].\]Hence, the quotient operator \;$\left[\,\left(\,U\,K\,\right)^{\,\ast} \;/\; \left(\,U\,S_{\Lambda}\,U^{\,\ast}\,\right)^{\,\dfrac{1}{2}}\,\right]$\; is bounded.

$(III) \,\Rightarrow\, (I)$\; Suppose that the quotient operator \;$\left[\,\left(\,U\,K\,\right)^{\,\ast} \;/\; \left(\,U\,S_{\Lambda}\,U^{\,\ast}\,\right)^{\,\dfrac{1}{2}}\,\right]$\; is bounded.\;Then for each \,$f \,\in\, H$, \,$\exists\, \,B \,>\, 0$\; such that
\begin{equation}\label{eq3.7} 
\left\|\,\left(\,U\,K\,\right)^{\,\ast}\;f\,\right\|^{\,2} \;\leq\; B\; \left\|\,\left(\,U\,S_{\Lambda}\,U^{\,\ast}\,\right)^{\dfrac{1}{2}}\,f\,\right\|^{\,2}.
\end{equation}
Now, by Theorem (\ref{th1}), for each \,$f \,\in\, H$, we have
\[\sum\limits_{\,j \,\in\, J}\,v_{j}^{\,2}\, \left\|\,\Lambda_{j}\,P_{\,W_{j}}\,U^{\,\ast}\,P_{\,U\,W_{j}}\,(\,f\,) \,\right\|^{\,2} \;=\; \sum\limits_{\,j \,\in\, J}\,v_{j}^{\,2}\, \left\|\,\Lambda_{j}\,P_{\,W_{j}}\,(\,U^{\,\ast}\,f\,) \,\right\|^{\,2}\]
\[=\; \left<\,S_{\Lambda}\,\left(\,U^{\,\ast}\,f\,\right) \;,\; U^{\,\ast}\,f\,\right> \,=\, \left\|\,\left(\,U\,S_{\Lambda}\,U^{\,\ast}\,\right)^{\dfrac{1}{2}}\,f\,\right\|^{\,2} \,\geq\; \dfrac{1}{B}\; \left\|\,\left(\,U\,K\,\right)^{\,\ast}\;f\,\right\|^{\,2}\; \;[\;\text{by (\ref{eq3.7})}\;]\,.\]
Also, since \,$\Lambda$\; is a \,$K$-$g$-fusion frame, \,$\exists$\, \,$C \,>\, 0$\, such that
\[\sum\limits_{\,j \,\in\, J}\,v_{j}^{\,2}\, \left\|\,\Lambda_{j}\,P_{\,W_{j}}\,U^{\,\ast}\,P_{\,U\,W_{j}}\,(\,f\,) \,\right\|^{\,2} \,=\, \sum\limits_{\,j \,\in\, J}\,v_{j}^{\,2}\, \left\|\,\Lambda_{j}\,P_{\,W_{j}}\,(\,U^{\,\ast}\,f\,) \,\right\|^{\,2}\]
\[\hspace{5cm}\leq\; C\; \left\|\,U^{\,\ast}\,f\,\right\|^{\,2} \;\leq\; C\; \|\,U\,\|^{\,2}\; \|\,f\,\|^{\,2}\; \;\forall\; f \,\in\, H.\]
Hence, \,$\Gamma$\; is a \,$U\,K$-$g$-fusion frame.\;This completes the proof. 
\end{proof}

\end{document}